\documentclass[12pt,longbibliography]{article}

\usepackage[utf8]{inputenc}

\usepackage[T1]{fontenc} 
\usepackage{bm}

\usepackage[a4paper, margin=2.7cm]{geometry}

\usepackage{amsmath, amsthm, amsfonts, amssymb}
\usepackage{mathrsfs}           
\usepackage{color}
\usepackage[normalem]{ulem}

\usepackage[colorlinks=true,linkcolor=blue,citecolor=blue,urlcolor=blue,breaklinks]{hyperref}

\usepackage{bbm} 

\numberwithin{equation}{section}

\usepackage{breakurl}
\usepackage{natbib}
\usepackage{url}
\def\BB{{\cal B}}

\def\LL{{\cal L}}
\def\N{\mathbb{N}}
\def\R{\mathbb{R}}

\def\D{\mathcal{D}}
\def\M{\mathcal{M}}
\def\P{\mathcal{P}}
\def\cN{\mathcal{N}}
\def\cU{\mathcal{U}}

\DeclareMathOperator{\supp}{supp}

\def\d{\,\mathrm{d}}

\def \ddt{\frac{\mathrm{d}}{\mathrm{d}t}}
\def \ddt{\frac{\mathrm{d}}{\mathrm{d}t}}

\def\:{\colon}

\def\1{\mathbbm{1}}
\def\NormT{|\hskip-0.04cm|\hskip-0.04cm|}

\newcommand{\beqn}{\begin{equation}}
\newcommand{\eeqn}{\end{equation}}
\newcommand{\bear}{\begin{eqnarray}}
\newcommand{\eear}{\end{eqnarray}}
\newcommand{\bean}{\begin{eqnarray*}}
\newcommand{\eean}{\end{eqnarray*}}

\newcommand{\be}{\begin{equation}}
\newcommand{\ee}{\end{equation}}
\newcommand{\ba}{\begin{aligned}}
\newcommand{\ea}{\end{aligned}}

\newcommand{\bal}{\begin{aligned}}
\newcommand{\eal}{\end{aligned}}

\def\Nt{|\hskip-0.04cm|\hskip-0.04cm|}

\def\eps{{\varepsilon}}

\def\BBB{{\mathscr B}}

\def\AA{{\mathcal A}}
\def\BB{{\mathcal B}}
\def\CC{{\mathcal C}}
\def\XX{{\mathcal X}}

\newcommand{\wto}{\rightharpoonup}

\newcommand{\Black}{\color{black}}

\newtheorem{thm}{Theorem}[section]
\newtheorem{cor}[thm]{Corollary}
\newtheorem{lem}[thm]{Lemma}

\newtheorem{hyp}{Hypothesis}

\theoremstyle{definition}

\theoremstyle{remark}
\newtheorem{rem}[thm]{Remark}
\theoremstyle{example}

\def\thetitle{Harris-type results on geometric and subgeometric
  convergence to equilibrium for stochastic semigroups}

\def\theauthor{José A. Cañizo \& Stéphane Mischler}

\hypersetup{pdftitle={\thetitle}}
\hypersetup{pdfauthor={\theauthor}}

\title{\thetitle}
\author{\theauthor}

\date{October 2021}

\AtEndDocument{\bigskip\bigskip{\footnotesize
  \noindent
  \textsc{José A. Cañizo. IMAG \& Departamento de
    Matemática Aplicada, Universidad de Granada, Avenida de
    Fuentenueva S/N, 18071 Granada,
    Spain.}
  \\
  \textit{Email address}: \texttt{\href{mailto:canizo@ugr.es}{canizo@ugr.es}}
  \par
  \addvspace{\medskipamount}
  \noindent
  \textsc{Stéphane Mischler. CEREMADE, UMR CNRS 7534. Université Paris-Dauphine, Place du Maréchal de Lattre de
    Tassigny 75775 Paris Cedex 16, France.}
  \\
  \textit{Email address}: \texttt{\href{mailto:mischler@ceremade.dauphine.fr}{mischler@ceremade.dauphine.fr}}
}}

\begin{document}

\maketitle

\begin{abstract}
  We provide simple and constructive proofs of Harris-type theorems on
  the existence and uniqueness of an equilibrium and the speed of
  equilibration of discrete-time and continuous-time stochastic
  semigroups.  Our results apply both to cases where the relaxation
  speed is exponential (also called geometric) and to those with no
  spectral gap, with non-exponential speeds (also called
  subgeometric). We give constructive estimates in the subgeometric
  case and discrete-time statements which seem both to be new. The
  method of proof also differs from previous works, based on semigroup
  and interpolation arguments, valid for both geometric and
  subgeometric cases with essentially the same ideas. In particular,
  we present very simple new proofs of the geometric case.
\end{abstract}

\tableofcontents

\newpage
\section{Introduction}
\label{sec:intro}

\paragraph{Aim of the work}
 
The study of convergence to equilibrium of continuous or discrete
semigroups which preserve mass and positivity is central in the theory
of Markov processes and partial differential equations (PDEs).
Several results ensuring geometric (i.e.~exponential) or subgeometric
(for instance, polynomial) convergence to equilibrium in weighted
total variation norms for a broad family of processes are known as
\emph{Harris-type theorems} (or also sometimes as
\emph{Meyn-Tweedie-type theorems}).  They have been widely developed
during the last three decades and have seen a broad range of
applications to probability and PDEs problems.

Harris-type theorems concern the trajectories of this kind of
semigroups. In both the geometric and subgeometric cases, they
establish the existence of an equilibrium (often called stationary
state or invariant measure depending on the context) and a speed of
convergence of trajectories to it. The main assumptions of this type
of theorems are
\begin{equation*}
  \text{(i) a  \emph{strong positivity}, \emph{irreducibility} or \emph{coupling} condition,} 
\end{equation*}
as well as 
\begin{equation*}
  \text{(ii) a \emph{confinement} or \emph{Foster-Lyapunov} condition.}
\end{equation*}
The latter condition determines whether the speed of convergence is
geometric or subgeometric.

\smallskip

Our purpose in this paper is to establish some theorems of this type in
both the geometric and the subgeometric situations using elementary
semigroup tools, and avoiding some usual probabilistic arguments such
as estimates of the return time of a process to a given set.

Let us describe our results a bit more precisely by considering the
typical case of a discrete stochastic process associated to an
operator $S$, which must be positive and mass-preserving, defined on a
space of measures with a weighted total variation norm. Under both
conditions (i) and (ii), we will be able to exhibit two convenient
norms $\NormT \cdot \NormT$, $ \| \cdot \|_*$ and a scalar
$\alpha > 0$ such that \beqn\label{intro:estimfonda} \NormT S \nu
\NormT \le \NormT \nu \NormT - \alpha \| \nu \|_*, \eeqn for any
measure $\nu$ with vanishing mass.  The strict contraction estimate
\eqref{intro:estimfonda}, or a variant of it, is then used in order to
prove the existence of a positive equilibrium $\mu^*$ with unit mass
associated to the operator $S$ such that $ \| \mu^* \|_* < \infty$.
The same estimate can also used in order to prove the uniqueness of
this equilibrium under slightly stronger assumptions, for example that
$S$ is of Feller-type.  When $ \| \cdot \|_*$ is equivalent to
$\NormT \cdot \NormT$, which holds true when the confinement condition
(ii) is strong enough, then one easily deduces from
\eqref{intro:estimfonda} a geometric convergence of the sequence
$(S^n \nu)$ to $0$ for any $\nu$ with vanishing mass. Under weaker
confinement conditions, the norm $ \| \cdot \|_*$ is strictly
dominated by $\NormT \cdot \NormT$ and only a subgeometric convergence
of the sequence $(S^n \nu)$ to $0$ is established.  A version of these
ideas for continuous semigroups will be also deduced from the
analysis of the discrete case.

\smallskip

Our approach is inspired by the proof of Harris' result in the
geometric case by \citet{HM2008}, which uses mass transportation
metrics. Our proof is a simplification of these ideas which avoids the
use of mass transportation arguments, and can be adapted to the
subgeometric case as well.  Our result gives an alternative proof to
the geometric decay estimate of \citet{HM2008} and can be adapted to
give subgeometric decay rates for discrete semigroups under weaker
confinement conditions. In the continuous-time case we can recover
similar subgeometric decay rates as in
\citet*{GDF2009,Hairer2016notes}. We emphasize that the statements
apply to mass-preserving semigroups and give explicit versions for
them, since this is a common setting in PDE.

\paragraph{Previous contributions}

The ergodicity and stability theory of Markov processes has been
widely developed since the pioneering works of \cite{Doblin} and
\citet{Harris1956}, the last one giving name to these results, though
it considers only the existence of equilibrium and does not mention
speed of convergence towards
it.
A good exposition of this type of results is given in
\citet{MeynTweedie}, and a nice introduction in the setting of Markov
chains can be found in \citet[Chapter 2]{Stroock2005}.

An
important development of the theory is due to
\cite{MeynTweedieI,MeynTweedieII,MeynTweedieIII,10.2307/2245077}.  A
simplified statement and a proof using mass transportation distances
was given by \citet{HM2008}, motivated by the application to
stochastic PDEs \citep{HM2008-NS}.  Recent related results for
non-conservative semigroups have been reported by \citet*{BCG2017},
and applications to models for the electrical activity of groups of
neurons can be found in \citet{DG2017, CY2018}.  Recent works dealing
with applications to Fokker-Planck equations and related models are
due to \citet*{MR3892316,MR3939573, MR4069622} and \citet{MR4049394}. We
also mention applications to the study of hypocoercivity for kinetic
equations and fragmentation-type equations \citet{CCEY2020, CGY2020}.

On the other hand, this type of theorems has been extended in several
works to the case of semigroups with no spectral gap, for which the
speed of convergence to equilibrium is subgeometric (slower than
exponential). Probabilistic results of this kind can be found in
\citet*{TT94, DFMS04}. We highlight \citet*{GDF2009}, where a result
for the continuous-time case was given, and serves as a model for the
type of results we wish to obtain in the present paper. An exposition
of this same result which also uses probabilistic arguments can be
found in unpublished notes by \citet{Hairer2016notes}, see also \cite{BernouSem}. Subgeometric
convergence rates have also been studied for classical models as the
Fokker-Planck equation, and the Boltzmann equation and its relatives;
for this we refer to \cite{KMN2015, MR3625186} and the references
therein, and the classical papers by \cite{MR576265,MR575897}. We also
mention the recent works by \cite{BernouFournier,Bernou2019}, where
convergence to equilibrium for a collisionless model of a gas is
investigated, the last one using techniques related to the present paper.

\paragraph{Definitions and notation}

We fix a measurable space $(\Omega, \mathcal{E})$ throughout. We
denote by $\M$ the set of finite signed measures on $\Omega$, and by
$\P$ the set of probability measures on $\Omega$. We also call $\cN$
the linear subspace of $\M$ consisting of zero mean measures (that is,
$\nu \in \cN$ if $\nu \in \M$ and $\nu(\Omega) = 0$).

We usually denote by $\int f \mu$ the integral
of a function $f$ with respect to a measure $\mu \in \M$, omitting the
domain of integration $\Omega$, and preferring this notation to the
also common $\int f \d \mu$. The positive and negative parts of a measure $\mu \in \M$ (with the
usual Hahn-Jordan decomposition) are denoted respectively by $\mu_+$,
$\mu_-$, so that $\mu = \mu_+ - \mu_-$ and $|\mu| := \mu_+ + \mu_-$. 
The total variation norm of a measure
$\mu \in \M$ is denoted by $\|\mu\| := \int |\mu|.$

A \emph{stochastic operator}
is a linear operator $S \: \M \to \M$ which leaves $\P$ invariant
(that is, a linear operator which preserves mass and positivity).  A
\emph{stochastic semigroup} is a family $(S_t)_{t \in [0,+\infty)}$ of
stochastic operators $S_t \: \M \to \M$ such that $S_0 = I$ and
$S_t \circ S_s = S_{t+s}$ for all $s, t \geq 0$. It is worth
emphasizing that we do not impose here any continuity assumption on the
trajectory $t \mapsto S_t \mu$ for a given $\mu \in \M$, so this
definition of stochastic semigroup is quite weak. Our results on the
Harris theorem for geometric decay require no further regularity
conditions on the semigroup; see Section \ref{sec:Harris}.

These objects are dual to the more classical definition of
Markov-Feller operators and semigroups. Whenever we need to consider
Markov-Feller semigroups we will always assume that $\Omega$ is a
locally compact and separable metric space, and call $C_0(\Omega)$ the
space of continuous functions which converge to $0$ at infinity (the
completion in the supremum norm of $C_\mathrm{c}(\Omega)$, the set of
continuous compactly supported functions on $\Omega$). The space
$C_0(\Omega)$ is a Banach space when endowed with the supremum norm,
and its dual is $\M$ (with the total variation norm) by the
Markov-Riesz representation theorem\footnote{For any locally compact
  Hausdorff topological space $\Omega$, the dual of $C_0(\Omega)$ is
  the set of finite Radon measures on $\Omega$; in our setting in
  which $\Omega$ is additionally a separable metric space, the set of
  Radon measures is just the set of finite measures---see for example
  \citet[Theorem 7.8]{Folland1999}.}. In this setting, a
\emph{Markov-Feller operator} $P$ is a linear and continuous operator
on $C_0(\Omega)$ which is positive ($P \varphi \ge 0$ if
$\varphi \ge 0$) and preserves constants ($P \varphi_n \nearrow 1$ if
$\varphi_n \nearrow 1$, with convergence understood in a pointwise
sense). A \emph{Markov-Feller semigroup} $(P_t)_{t \ge 0}$ is a
strongly continuous semigroup of {Markov-Feller} operators on
$C_0(\Omega)$.  If $P$ is a Markov-Feller operator then its dual
$S := P^*$ is a stochastic operator, and in that case we will say that
$S$ is \emph{of Feller type}. Similarly, if $(P_t)_{t \geq 0}$ is a
Markov-Feller semigroup then the semigroup $(S_t)_{t \geq 0}$ defined
by $S_t := P^*_t$ is a stochastic semigroup. In that case, we say that
$(S_t)_{t \geq 0}$ is of Feller type and we denote by ${L}$ the
generator of $(P_t)$ in the sense of semigroups.
We note that for a Feller type stochastic semigroup $(S_t)_{t \geq 0}$
the trajectory $t \mapsto S_t \mu$ is weakly-$*$ continuous for any
given $\mu \in \M$ (that is, the trajectory is continuous in the
weak-$*$ topology of $\M$, viewed as the dual of $C_0(\Omega)$), but
is does not need to be continuous in the total variation norm.  Let us
emphasize
that definitions of ``Feller'' for an operator or a semigroup vary
slightly in the literature; in some references a Markov-Feller
operator is defined as an operator on $C_b(\Omega)$, the set of
continuous and bounded functions
\citep[Definition~1.9]{Hairer2016notes}, defined through an integral
formula involving a transition kernel. We will not consider this
latter case here, but rather we will use the concept of Feller-type
stochastic operators and semigroups defined by duality from
$C_0(\Omega)$, for which some simplifications occur. However, we
emphasize that many of the results we state work with the minimal
assumption of a stochastic operator or semigroup.

\smallskip 

For a measurable (weight) function
$V \: \Omega \to [1,+\infty)$, we denote by $\M_V$ the subspace of
finite signed measures $\mu$ on $\Omega$ such that
$$
\| \mu \|_V := \int_\Omega V |\mu| < \infty,
$$
and write $\P_V := \M_V \cap \P$ for the set of probability measures
for which $\| \mu \|_V < +\infty$, and similarly
$\cN_V := \M_V \cap \cN$ is the set zero-mean measures with
$\| \mu \|_V < +\infty$.  We say that $S$ is a \emph{stochastic
  operator on $\M_V$} if it is a stochastic operator on $\M$, one can
restrict $S \: \M_V \to \M_V$, and this restriction is bounded in the
$\|\cdot\|_V$ norm.  Similarly, we say that $(S_t)_{t \geq 0}$ is a
\emph{stochastic semigroup on $\M_V$} if it is a stochastic semigroup
on $\M$ and satisfies a growth estimate
\begin{equation} \label{eq:V-semigroup-growth}
  \|S_t \mu \|_V \leq C_V e^{\omega_V t} \|\mu\|_V,
\end{equation}
for all $\mu \in \M_V$ and all $t \ge 0$, and for some constants
$C_V \ge 1$, $\omega_V \ge 0$.

\paragraph{Plan of the paper}

The paper is organized as follows.  We first prove in Section
\ref{sec:Doeblin} a simple statement, sometimes known as Doeblin's
theorem. The statement and proof of the geometric version of Harris'
theorem is next given in Section \ref{sec:Harris}.
Sections~\ref{sec:subgeometric-discrete} and
\ref{sec:time-continuous} are then devoted to our versions of
Harris' theorem is the case of subgeometric operators and semigroups
respectively. In the final section~\ref{sec:equilibrium}, the proof of
the existence of an  equilibrium (but not its uniqueness nor its
stability) is established only assuming a Lyapunov condition (but
without any irreducibility assumption).

\section{Doeblin's theorem}
\label{sec:Doeblin}
 
  In this section, we present a basic and well-known  result in the theory of Markov
processes sometimes known as \emph{Doeblin's theorem}, which is a
particular case of the Harris theorem presented in the next
section. We include it since the proof is extremely simple and
contains ideas that are used in later proofs. The argument is widely
known, and we were made aware of it through \citet{Gabriel}.

It is well known that stochastic operators are non-expansive mappings
(or contractions in the non-strict sense) in the measure space $\M$,
namely
 \begin{equation}
   \label{eq:Doeblin-contrac1}
   \| S\mu  \|   \leq \|\mu \|, 
 \end{equation}
 for all measures $\mu \in \M$. The proof of this fact is simple and
 instructive. We introduce the Hahn-Jordan decomposition
 $\mu = \mu_+ - \mu_-$, $0 \le \mu_\pm \in \M$, which ensures $|\mu| =
 \mu_+ + \mu_-$, and we write
 $$
 |S\mu| = |S\mu_+ - S\mu_- | \le S\mu_+ +  S\mu_- = S|\mu|,
 $$
 by using the linearity and the positivity of $S$. We then immediately
 deduce \eqref{eq:Doeblin-contrac1} by integrating the last inequality and by using that $S$ is mass preserving.

 \smallskip Doeblin's Theorem states that under some very strong
 positivity or irreducibility condition the above non-expansive
 property becomes a (strict) contraction property on the set $\cN$ of
 zero mean measures:

\begin{thm}[Doeblin's theorem]
  \label{thm:Doeblin}
  Let $S \: \M \to \M$ be a stochastic operator satisfying that there exist
  $0 < \alpha < 1$ and $\eta \in \P$ such that 
  \begin{equation}
    \label{eq:Doeblin}
  S \mu \geq \alpha \eta, 
    \qquad
    \text{for all $\mu \in \P$.}
  \end{equation}
  Then $S$ has a unique stationary state $\mu^* \in \P$ which is
  exponentially stable, and more generally
  \begin{equation}
    \label{eq:Doeblin-contracn}
    \| S^n\nu \|
    \leq
\gamma^n \|\nu\| \qquad
    \text{for all $\nu \in \cN$ and $n \in \N$,}
  \end{equation}
  with $\gamma := 1-\alpha \in (0,1)$.
  \end{thm}

It is worth emphasizing that for any $\mu \in \P$, we deduce from \eqref{eq:Doeblin-contracn} that 
   \begin{equation*}
     \|S^n \mu - \mu^* \|
     \leq
   \gamma^n \|\mu - \mu^* \|, 
     \qquad
     \text{for all   $n \in \N$,}
 \end{equation*}
 and thus the exponential asymptotic stability of the equilibrium
 $\mu^*$. 
  
\begin{proof}[Proof of Theorem \ref{thm:Doeblin}]
The proof is based on an improvement of \eqref{eq:Doeblin-contrac1} which writes 
 \begin{equation}
    \label{eq:Doeblin-contrac}
    \| S \nu \|
    \leq
    \gamma \|\nu\| \qquad
    \text{for all $\nu \in \cN$,}
  \end{equation}
  with $\gamma := 1-\alpha$.  In order to prove
  \eqref{eq:Doeblin-contrac}, we observe that because of the 
    Doeblin condition \eqref{eq:Doeblin} applied to
  $S \bigl( { \nu_\pm / \| \nu_\pm\|} \bigr)$ and the fact that the
  integrals of $\nu_+$ and $\nu_-$ are equal for $\nu \in \cN$, it
  holds
  \begin{equation*}
    S\nu_\pm \ge   \alpha \eta  
    \int \nu_{\pm}= r \, \eta, \quad r := \alpha \| \nu \|/2.
  \end{equation*}
  Similarly as in the proof  of \eqref{eq:Doeblin-contrac1}, we may deduce
  \bean |S \nu | &=& |S \nu_+ - r\eta - S\nu_- + r \eta|
  \\
  &\leq& |S \nu_+ - r\eta| + |S\nu_- - r \eta|
  \\
  &=& S \nu_+ - r\eta + S\nu_- - r \eta \ = \ S |\nu| - 2r\eta, \eean
  and integrating this, we get 
  \begin{equation*}
    \| S \nu \| \leq
    \| S |\nu|  \| - 2 r  \|  \eta \| = \| \nu  \| - 2 r
    = (1-\alpha) \| \nu \|. 
  \end{equation*}
  That is exactly inequality \eqref{eq:Doeblin-contrac}, from which
  \eqref{eq:Doeblin-contracn} immediately follows.
  
  \smallskip In order to prove the existence and uniqueness of an
  equilibrium,
  we fix $\mu_0 \in \P$, and we define recursively
  $\mu_k := S \mu_{k-1}$ for any $k \ge 1$.  Thanks to
  \eqref{eq:Doeblin-contrac}, we get
  $$
  \sum_{k=1}^\infty \| \mu_k - \mu_{k-1} \| \le \sum_{k=0}^\infty \gamma^k \| \mu_1- \mu_0 \| < \infty, 
  $$
  so that $(\mu_k)$ is a Cauchy sequence in $\P$. We set
  $\mu^*:= \lim \mu_k \in \P$ which is a stationary state, as seen by
  passing to the limit in the equation $\mu_k = S \mu_{k-1}$, and
  which is unique in $\P$ thanks to
  \eqref{eq:Doeblin-contrac}.
\end{proof}


\section{Harris's theorem}
\label{sec:Harris}

We extend Doeblin's results presented in the previous
section to the case when only a weaker version of Doeblin's
positivity condition \eqref{eq:Doeblin} holds, together with a
\emph{Lyapunov condition}.  An important motivation is that the
Doeblin condition \eqref{eq:Doeblin} is indeed too restrictive and
somehow limited to compact spaces.  As a matter of fact, when
$\Omega = \R^d$ for instance and $S$ is a stochastic operator of
Feller type, there exists a sequence $(\mu_{0n})$ in $\P$ such that
$\mu_{0n} \wto 0$ (in the weak-$*$ sense of measures; take for example
$\mu_{0n} := \delta_{x_n}$ with $|x_n| \to +\infty$). Since $S$ is
continuous in the weak-$*$ topology due to $S$ being of Feller type,
also $S \mu_{n0} \wto 0$ and the Doeblin condition
\eqref{eq:Doeblin} cannot hold.  However,
the condition \eqref{eq:Doeblin} may still hold if the semigroup is
not of Feller type in our sense; an example on $\Omega = [0,+\infty)$
is the renewal equation found in \cite{Gabriel}. For many
applications, one must thus weaken the positivity condition
\eqref{eq:Doeblin}. This forces us to add a localization or
confinement condition and work in a weighted space.
 
\medskip The following assumptions will be used in Harris's theorem
below.  In all of this section, $V \: \Omega \to [1,+\infty)$ denotes
a measurable function, that we will call in the sequel a {\it Lyapunov
  or  weight function.}

\begin{hyp}[Operator Lyapunov condition]
  \label{hyp:op_Lyapunov}
  An operator $S$ satisfies an {\it operator Lyapunov condition} with
  Lyapunov function $V$ if there exist $0 < \gamma_L < 1$ and
  $K \geq 0$ such that
  \begin{equation}
    \label{eq:Lyapunov}
    \|S \mu \|_V \leq \gamma_L \|\mu\|_V + K \|\mu\|,
    \qquad
    \text{for   $\mu \in \mathcal{M}_V$.}
  \end{equation}
\end{hyp}

\begin{hyp}[Harris condition]
  \label{hyp:Harris}
  An operator $S$ satisfies a {\it Harris condition} on a set
  $\CC \subseteq \Omega$ if there exist $0 < \alpha < 1$ and
  $\eta \in \P$ such that
  \begin{equation}
    \label{eq:Doeblin-Harris}
    S \mu \geq \alpha \eta \int_\CC \mu, 
    \qquad
    \text{for all $0 \le \mu \in \mathcal{M}$.}
  \end{equation}
\end{hyp}

In other words, Hypothesis \ref{hyp:Harris} states that the
Doeblin condition \eqref{eq:Doeblin} holds, but only for mesures $\mu$
supported on the set $\CC$.

\begin{hyp}[Local coupling condition]
  \label{hyp:local_coupling}
  An operator $S$ satisfies a {\it local coupling condition} with
  Lyapunov function $V$ if there exist $0 < \gamma_H < 1$ and $A> 0$
  such that
  \begin{equation} \label{eq:coupling} \Bigl( \nu \in
    \cN_V,
    \ \| \nu \|_V \le A \| \nu \| \Bigr) \quad \text{implies} \quad
    \| S \nu \| \leq \gamma_H \| \nu \| .
  \end{equation}
\end{hyp}
The term {\it local coupling condition} comes from the fact that it
implies that (and is in fact equivalent to)
$$
\Bigl(  x,y \in \Omega,   \  V(x) + V(y) \le A  \Bigr) \quad \text{implies} \quad
\|  S (\delta_x - \delta_y) \| \leq 2 \gamma_H,
$$
so that the distance between $S\delta_x$ and $S\delta_y$ is strictly
less than the distance between $\delta_x$ and $\delta_y$ under a
localisation condition on $x$ and $y$.  
  The following lemma shows that, roughly speaking, 
the Harris hypothesis \ref{hyp:Harris} implies the Local coupling Hypothesis  \ref{hyp:local_coupling}.

\begin{lem}[Harris implies local coupling]
  \label{lem:Harris&coupling}
  If $S$ satisfies the {\it Harris condition} (Hypothesis
  \ref{hyp:Harris}) on the set
  $\CC = \{ x \in \Omega \mid V(x) \leq R \}$ for some $R> 0$ and
  $0 < \alpha < 1$ then it satisfies the {\it local coupling
    condition} (Hypothesis \ref{hyp:local_coupling}) with any
  $A \in (0, R/2)$ and $\gamma_H := 1 - \alpha (1-2A/R) \in (0,1)$.
\end{lem}

\begin{proof}[Proof of Lemma~\ref{lem:Harris&coupling}]
  Under the Harris condition \eqref{eq:Doeblin-Harris} and when $\nu$
  satisfies the LHS hypotheses of condition \eqref{eq:coupling}, a
  sizeable part of the mass of $\nu_+$ and $\nu_-$ is in
  $\mathcal{C} := \{x \in \Omega \mid V(x) \leq R \}$, as can be seen
  from the bound
  \begin{equation*}
    \int_{\Omega \setminus \mathcal{C}} \nu_{\pm}
    \leq \frac{1}{R} \int V |\nu|
    \leq
    \frac{A}{R} \int |\nu|
    =
    \frac{2 A}{R} \int \nu_{\pm},
  \end{equation*}
  where we have used in a fundamental way that the masses of $\nu_+$ and
  $\nu_-$ are equal in the last line. That implies
  \begin{equation*}
    \int_{\mathcal{C}} \nu_{\pm} \geq
    \left(
      1 - \frac{2  A}{R}
    \right)
    \int \nu_{\pm}.
  \end{equation*}
  Because of the {\it Harris condition} \eqref{eq:Doeblin-Harris} and
  the fact that the mass of $\nu_+$ and $\nu_-$ are equal, it
  holds
  \begin{equation*}
    S\nu_\pm \ge   \alpha \eta  
    \left(
      1 - \frac{2  A}{R}
    \right)
    \int \nu_{\pm} =: r \, \eta,
  \end{equation*}
  with
  \begin{equation*}
    r := \alpha  \left( 1 - \frac{2  A}{R} \right) \int \nu_{\pm}
    = \frac{1-\gamma_H}{2} \| \nu \|.
  \end{equation*}
  Repeating the proof of Theorem \ref{thm:Doeblin}, it holds then
   \begin{equation*}
    \| S \nu \| \leq  \| \nu  \| - 2 r = \gamma_H \| \nu \|.
   \end{equation*}
  That is exactly inequality \eqref{eq:coupling}
  with $\gamma_H := 1-\alpha(1-2A/R)$.
\end{proof}

\begin{thm}[Harris's Theorem]
  \label{thm:Harris}
  Consider $S \: \M_V \to \M_V$ a stochastic operator which satisfies
  the operator Lyapunov condition (Hypothesis \ref{hyp:op_Lyapunov})
  and the local coupling condition (Hypothesis
  \ref{hyp:local_coupling}) with $K/A < 1-\gamma_L$, both with the
  same weight function $V$.  Then $S$ has a unique stationary state
  $\mu^* \in \P_V$ which is exponentially stable.  More generally,
  there exist $\gamma \in (0,1)$ and $C \in [1,\infty)$ such that
  \begin{equation}
    \label{eq:Harris-exponential-convergence}
    \| S^n\nu  \|_{V}
    \leq
    C \,  \gamma^n \|  \nu  \|_{V}, \qquad
    \text{for all $\nu \in \cN_V$ and $n \in \N$.}
  \end{equation}
\end{thm}
 
Due to Lemma \ref{lem:Harris&coupling}, the conclusion of Theorem
\ref{thm:Harris} applies also if $S$ satisfies the Lyapunov condition
\eqref{eq:Lyapunov} and the Harris condition \eqref{eq:Doeblin-Harris}
with $2 K / R \leq 1 - \gamma_L$.  With this result, we recover the
main result of \citet{HM2008} with a similar approach, except that we
work on the stochastic operator side rather than on the dual Markov
operator side.  In particular, and as in Doeblin's framework
  of Section~\ref{sec:Doeblin}, we deduce the exponential asymptotic
  stability of the equilibrium $\nu^*$ in $\M_V$, namely
   \begin{equation*}
     \|S^n \mu - \mu^* \|_V
     \leq
  C \gamma^n \|\mu - \mu^* \|_V, 
     \qquad
     \text{for all  $\mu \in \P_V$ and $n \in \N$.}
 \end{equation*}
 The theorem  doesn't  exclude the possibility  of other equilibria with infinite V-moment.

\begin{proof}[Proof of Theorem \ref{thm:Harris}]
  We introduce a new norm $\Nt \cdot \Nt_{V}$ on $\M_V$ defined by
 \begin{equation}\label{eq:newNormV}
\Nt  \mu  \Nt_{V}     := \|\mu\|+ \beta \| \mu \|_V, 
  \end{equation}
  for some $\beta > 0$ to be chosen later. Note that
  $\Nt \cdot \Nt_{V}$ and $\|\cdot\|_V$ are equivalent norms, with
\begin{equation*}
  (1 + \beta)^{-1}\Nt  \mu  \Nt_{V}   \leq \| \mu \|_V
  \leq \beta^{-1} \Nt  \mu  \Nt_{V} . 
\end{equation*}
We claim that there exist $\beta > 0$ small enough and $\gamma \in (0,1)$ such that 
  \begin{equation}
    \label{eq:Harris-contrac}
  \Nt  S \nu   \Nt_{V}    \leq  \gamma   \Nt  \nu  \Nt_{V},   \qquad
    \text{for all $\nu \in \cN_V$.}
  \end{equation}
Using  \eqref{eq:Harris-contrac}, we may then straightforwardly adapt the proof of Theorem~\ref{thm:Doeblin} in order to conclude to the
existence and uniqueness of  a stationary state  $\mu^* \in \P_V$  of $S$ and to the geometrical decay  
\eqref{eq:Harris-exponential-convergence} with $C := (1+\beta) / \beta$.

\smallskip
We may then focus on the proof of the contraction estimate \eqref{eq:Harris-contrac}. 
For that purpose, we take any $\nu \in \cN$ and estimate the norm $  \Nt  S \nu   \Nt_{V} $ in two alternative cases:

\Black

   \medskip\noindent \textbf{First case.} \emph{Contractivity for
    small $V$-moment.} When 
      \begin{equation}
    \label{eq:moment-cond-rev}
    \| \nu \|_V  < A  \| \nu \|, 
  \end{equation}
  the local coupling condition \eqref{eq:coupling} implies 
  \begin{align*}
    \|S \nu\| \le \gamma_H \|\nu\|.
   \end{align*}
   Together with the Lyapunov condition \eqref{eq:Lyapunov}, we  have
   \begin{align*}
  \Nt  S \nu   \Nt_{V}     
     &=
         \|S \nu\| + \beta       \|S \nu\|_V
     \\
     &\leq
            (\gamma_H+\beta K) \|\nu\|
            + \beta \gamma_L    \|S \nu\|_V
            \leq
            \gamma_1 \|\nu\|_\beta,
   \end{align*}
   with
   $$
   \gamma_1 := \max\{   \gamma_H + \beta K,  \gamma_L \}.
   $$
Choosing $\beta > 0$ small enough such that  $\beta K <  1 - \gamma_H$, we get $\gamma_1 < 1$ and that gives the contractivity property     \eqref{eq:Harris-contrac} in this
   case.

  \medskip\noindent \textbf{Second case.} \emph{Contractivity for large
    $V$-moment.} 
     Assume on the contrary that
   \begin{equation}
    \label{eq:moment-cond}
      \| \nu \|_V  \ge A  \| \nu \|.
  \end{equation}
  From \eqref{eq:Lyapunov} we deduce then
  \begin{equation*}
    \|S \nu \|_V \leq \gamma_L \|\nu\|_V + K \|\nu\|
    \leq
    (\gamma_L + K/A) \|\nu\|_V,
  \end{equation*}
  with $\gamma_L + K/A < 1$ by assumption. Together with
  \eqref{eq:Doeblin-contrac1}, we deduce
  \begin{align*}
    \Nt  S \nu   \Nt_{V}     
    &=
      \|S \nu\| + \beta       \|S \nu\|_V
    \\
    &\leq
      \| \nu\| + \beta    (\gamma_L + K/A)    \| \nu\|_V         
    \\
    &\leq
      (1-\beta \delta_0) \| \nu\| +  \beta    (\gamma_L + K/A + \delta_0)    \| \nu\|_V, 
  \end{align*}
  for any $\delta_0 \ge 0$, by using that $V \geq 1$ in the last
  inequality above.  We thus get
  \begin{align*}
    \Nt  S \nu   \Nt_{V}      \leq
    \gamma_2  \Nt  S \nu   \Nt_{V} ,
  \end{align*}
  with $\gamma_2 := \max(1-\beta \delta_0,\gamma_L + K/A + \delta_0)$.
  We get the contractivity property \eqref{eq:Harris-contrac} in this
  case by choosing $\delta_0 > 0$ small enough (and keeping the choice
  of $\beta > 0$ made in the previous case) so that
  $\gamma_2 \in (0,1)$.  The proof of \eqref{eq:Harris-contrac} is
  completed by setting $\gamma := \max \{ \gamma_1,\gamma_2 \}$.
 \end{proof}

\begin{rem}
  By following the above proof one can give an explicit expression of
  the constants. Because
   $$
   \gamma_L < \gamma_1 := 1 - {\beta \over 1 + \beta} (1 - \gamma_L - K/A),
   $$
   we have 
   $$
   \gamma = \max\bigl\{   \gamma_H + \beta K,  1 - {\beta \over 1 + \beta} (1 - \gamma_L - K/A) \bigr\}.
   $$
   We see then that the best choice of $\beta$ is the (uniquely
   defined) positive zero of the following second order polynomial
   equation
   $$
   K \beta^2 +   (K + b -a ) \beta - a = 0,
   $$
   with $ a := 1-\gamma_H > 0$, $b := 1-\gamma_L - K/A > 0$.  
 \end{rem}

 We end the section by presenting a different proof of Theorem
 \ref{thm:Harris}. The outcome is essentially the same, but we do not
 obtain as part of the argument the contractivity of a modified
 weighted total variation norm as in the previous proof. On the other
 hand, the result has an extremely short proof which makes the role of
 the assumptions very clear!

 \begin{proof}[Alternative proof of Theorem \ref{thm:Harris}]
   Given $\nu \in \cN_V$,
   we call
   \begin{equation*}
     v_n := \| S^n \nu \|_V,
     \qquad
     m_n :=  \| S^n \nu \|,
   \end{equation*}
   for integer $n \geq 0$. The Lyapunov condition
   \eqref{eq:Lyapunov} shows that
   \begin{equation}
     \label{eq:1}
     v_{n+1} \leq \gamma_L v_n + K m_n.
   \end{equation}
   The local coupling condition \eqref{eq:coupling} and the
   non-expansive mapping property \eqref{eq:Doeblin-contrac1} together
   imply
   \begin{equation*}
 \| S \nu \| \leq
     \begin{cases}
      \gamma_H \| \nu \| &\qquad \text{whenever} \quad \| \nu \|_V \le A \| \nu \|,
       \\
       \| \nu \|&\qquad \text{always.}
     \end{cases}
   \end{equation*}
   In particular,
   \begin{equation*}
    \| S \nu \| \leq
\gamma_H    \| \nu \|
     + \frac{1-\gamma_H}{A}  \| \nu \|_V
     \qquad \text{always,}
   \end{equation*}
   since the inequality can be checked to be true in the two cases
   $\| \nu \|_V \le A \| \nu \|$ and $\| \nu \|_V > A \| \nu \|$.
   Iterating this we get
   $m_{n+1} \leq \gamma_H m_n + \frac{1-\gamma_H}{A} v_n$. Together
   with \eqref{eq:1}, this gives the system
   \begin{align*}
     v_{n+1} &\leq \gamma_L v_n + K m_n,
     \\
     m_{n+1} &\leq \frac{1-\gamma_H}{A}  v_n +    \gamma_H  m_n,
   \end{align*}
   whose associated matrix is
   \begin{equation*}
     M := \left(
       \begin{matrix}
         \gamma_L & K
         \\
        \frac{1-\gamma_H}{A} & \gamma_H 
       \end{matrix}
     \right).
   \end{equation*}
   One can easily see that the condition for the eigenvalues of this
   matrix to be both strictly less than 1 is that $1-\gamma_L > K/A$,
   so that both $v_n$ and $m_n$ decay exponentially in $n$. Existence
   and uniqueness of an   equilibrium in $\P_V$ follow as before.
 \end{proof}


 \section{Subgeometric convergence for discrete-time
   semigroups}
\label{sec:subgeometric-discrete}

We now extend Harris's Theorem to cases in which a weaker form of
Lyapunov condition \eqref{eq:Lyapunov} holds true, with a slowing of
the speed of decay as a drawback.

\smallskip In all of this section, $V \: \Omega \to [1,+\infty)$ is a
measurable weight function, still referred to as a Lyapunov or just
weight function and $\varphi \: [1, +\infty) \to [1, +\infty)$ a concave
function with $\varphi(1)=1$ and
$\lim_{v \to +\infty} \varphi(v) / v = 0$.  The following assumption
generalizes the Lyapunov condition from Hypothesis
\ref{hyp:op_Lyapunov} and will be used in the subgeometric version of
Harris's theorem below.

\begin{hyp}[Weak operator Lyapunov condition]
  \label{hyp:op_Lyapunov_weak}
  A stochastic operator $S$ satisfies a
  \emph{weak Lyapunov condition} for $V$ and $\varphi$ if there
  exist $K > 0$ and $0 < \varsigma < 1$ such that
  \begin{equation}
    \label{eq:weak_Lyapunov_implicit}
    \| S \mu \|_{V} + \varsigma \| \mu \|_{\varphi(V)}
    \le
    \| \mu \|_{V} +  K \|\mu \|, \qquad \text{for all $\mu \in \mathcal{M}_{V}$}.
  \end{equation}
\end{hyp}

Let us make some observations. 

\begin{rem}\label{rem:weakLyap1}
Because $\varphi \: [1,+ \infty) \to [1,+\infty)$ is a  concave function,
 then $\varphi$ must be continuous and nondecreasing. 
 The continuity of $\varphi$ ensures $\varphi(V) \equiv \varphi \circ V$ is measurable.
 The asymptotic 
condition $\lim_{v \to +\infty} \varphi(v) / v = 0$ ensures that we are not in the  framework of Section
\ref{sec:Harris} since Hypothesis \ref{hyp:op_Lyapunov} does not hold. 
\end{rem}

\begin{rem}\label{rem:FellerOp}
  If $S$ is a Feller-type stochastic operator and $P$ is the associated Markov-Feller operator on
  $\mathcal{C}_0(\Omega)$ such that $P^* = S$ then, by duality,
  Hypothesis \ref{hyp:op_Lyapunov_weak} is equivalent to the property
  \begin{equation*}
     P V + \varsigma \varphi(V) \leq V + K,
  \end{equation*}
  which is perhaps more often found in the literature  (see for instance, \citet[Theorem~3.3 (i)]{GDF2009}). 
 It is worth emphasizing here that a possible definition of the function $PV$ is $PV := \lim P(V \varphi_n) \in [0,\infty]$, 
where $(\varphi_n)$ is a nonnegative sequence of $C_0(\Omega)$ such that $\varphi_n \nearrow 1$, 
  which belongs to $\LL^\infty_{loc}(\Omega)$ because of the above Lyapunov property.

\end{rem}

\subsection{Existence of an equilibrium}
\label{sec:existence}

We now show that under weak Lyapunov and coupling conditions one can
build a norm $\Nt \cdot \Nt_{V}$ equivalent to $\|\cdot\|_{V}$ for
which our stochastic operator $S$ is a contraction, in a quantitative
sense. This will be used for existence and uniqueness results, and later for
obtaining decay rates.

\begin{lem}
  \label{lem:nonexpansive}
  Consider a stochastic operator $S$ such
  that
  \begin{enumerate}
  \item $S$ satisfies a weak Lyapunov condition (Hypothesis
    \ref{hyp:op_Lyapunov_weak}) associated to functions $V$, $\varphi$ and 
    constants $K, \varsigma$.
    
  \item For some integer $N \geq 1$, the operator $S^N$ satisfies a
    local coupling condition (Hypothesis \ref{hyp:local_coupling}) for
    $\varphi(V)$,  with constant $A > K/\varsigma$. \Black
  \end{enumerate}
  Then for any $\nu \in \mathcal{N}_V$, there exists an integer
  $n$ with $N \leq n \leq 2N-1$ such that
  \begin{equation}
    \label{eq:subHarris-nonexp}
    \Nt S^n \nu \Nt_{V} + \alpha \sum_{k=0}^{n-1} \| S^k \nu \|_{\varphi(V)}
    \leq \Nt \nu \Nt_{V},
    \qquad \text{for all $\nu \in \cN_{V}$,}
  \end{equation}
  where
  \begin{equation*}
    \Nt \mu \Nt_{V} := \| \mu \| +  \beta \|\mu\|_{V},
    \qquad \text{for $\mu \in \mathcal{M}_{V}$,}
  \end{equation*}
and  with 
  \begin{equation}
    \label{eq:betadelta}
    \beta := (1-\gamma_H)/(KN),
    \quad \alpha := \beta (\varsigma - K/A) > 0.
  \end{equation}
\end{lem}

\begin{proof}[Proof of Lemma~\ref{lem:nonexpansive}]
  We fix $\nu \in \cN_{V}$ and denote $\nu_k := S^k \nu$ for all
  integer $k \geq 0$  and we set $V_0 := \varphi(V)$. 
  We observe that if for a given $k$, we have
  \begin{equation}
    \label{eq:V0-large}
    \|\nu_k\|_{V_0} \geq A \|\nu_k\|,
  \end{equation}
  then this inequality and the weak Lyapunov condition in Hypothesis
  \ref{hyp:op_Lyapunov_weak} imply
  \begin{equation*}
    \| \nu_{k+1} \|_V  \leq \| \nu_{k} \|_V
    - \big(\varsigma - \frac{K}{A}\big) \| \nu_k\|_{V_0},
  \end{equation*}
  where the quantity $\varsigma - K/A > 0$ by hypothesis, which allows us
  to carry out the argument. Multiplying by $\beta$, using that
  $\alpha = \beta (\varsigma - K/A)$, and the contractivity
  $\| \nu_{k+1} \| \leq \| \nu_k \|$, we have
  \begin{equation*}
    \beta \| \nu_{k+1} \|_V + \| \nu_{k+1} \|
    \leq
    \beta \| \nu_{k} \|_V
    - \alpha \| \nu_k\|_{V_0} + \|\nu_k\|,
  \end{equation*}
  that is
  \begin{equation}
    \label{eq:contractive-norm}
    \Nt \nu_{k+1} \Nt_V 
    \leq \Nt \nu_{k} \Nt_V - \alpha \| \nu_{k} \|_{V_0}.
  \end{equation}
  Now we have two cases:

  \smallskip\noindent {\sl Case 1. } If \eqref{eq:V0-large} holds for
  all integer $k$ with $0 \leq k \leq n-1$, then we directly obtain
  \eqref{eq:subHarris-nonexp} by iterating the difference inequality
  \eqref{eq:contractive-norm}.
  
  \smallskip\noindent {\sl Case 2. } If \eqref{eq:V0-large} fails for
  some $k$ in $\{0, \dots, n-1\}$, then take $k^*$ the smallest
  integer in this range in which the condition fails. Then we may use
  \eqref{eq:contractive-norm} for $0 \leq k < k^*$ and obtain
  \begin{equation}
    \label{eq:2}
    \Nt \nu_{k^*} \Nt_{V} 
    \leq \Nt \nu \Nt_{V} - \alpha \sum_{k=0}^{k^*-1} \| \nu_k \|_{V_0}.
  \end{equation}
  Define now $n := N + k^*$ in this case. Using
  that $S^N$ satisfies the coupling condition, we have
  \begin{equation}
    \label{eq:3}
    \| \nu_n \| \leq \gamma_H \|\nu_{k^*}\|.
  \end{equation}
  On the other hand, we may use the weak Lyapunov condition and the
  fact that $k \mapsto \| \nu_k \|$ is
  nonincreasing to get
  \begin{equation*}
    \|\nu_{k+1}\|_V \leq \|\nu_k\|_V - \varsigma \|\nu_{k}\|_{V_0}
    + K \|\nu_{k^*}\|,
  \end{equation*}
  for all $k = k^*, \dots, n-1$. Summing the inequality in this range,
  we get
  \begin{equation*}
    \|\nu_n\|_V  \leq
    \|\nu_{k^*}\|_V - \varsigma \sum_{k=k^*}^{n-1} \|\nu_k\|_{V_0}
    + NK \|\nu_{k^*}\|.
  \end{equation*}
  Multiplying by $\beta$ and adding $\|\nu_n\|$ to complete $\Nt \nu_n
  \Nt$ on the left hand side,
  we deduce
  \begin{equation*}
    \beta \|\nu_n\|_V + \| \nu_n\|  \leq
    \beta \|\nu_{k^*}\|_V - \beta \varsigma \sum_{k=k^*}^{n-1} \|\nu_k\|_{V_0}
    + \beta NK \|\nu_{k^*}\| + \| \nu_n\|.
  \end{equation*}
  Using \eqref{eq:3} and reorganising terms, we conclude with 
  \begin{multline*}
    \Nt \nu_n \Nt_V + \alpha \sum_{k=k^*}^{n-1} \|\nu_k\|_{V_0}
    \leq
    \beta \|\nu_{k^*}\|_V  
    + (\beta NK + \gamma_H) \|\nu_{k^*}\|
    \\
    =
    \beta \|\nu_{k^*}\|_V + \|\nu_{k^*}\|
   = \Nt \nu_{k_*} \Nt_V
    \leq \Nt \nu \Nt_{V} - \alpha \sum_{k=0}^{k^*-1} \| \nu_k \|_{V_0},
  \end{multline*}
  where in the last inequality we have used \eqref{eq:2}. This shows
  the result.
\end{proof}

\begin{thm}[Existence of equilibrium]
  \label{thm:existence}
  Consider a stochastic operator $S$ satisfying the same conditions as
  in Lemma \ref{lem:nonexpansive}.  Then there exists an equilibrium
  $\mu^* \in \P_{\varphi(V)}$.
\end{thm}

\begin{proof}[Proof of Theorem~\ref{thm:existence}]
  Take any $\mu_0 \in \P_V$ and define
  \begin{equation*}
    \nu_0 := S \mu_0 - \mu_0,
    \qquad
    \nu_{k} := S^k \nu_0, \quad k \geq 1.
  \end{equation*}
  From Lemma \ref{lem:nonexpansive}, we can find an increasing sequence $(n_i)_{i
    \geq 0}$ with $n_0 = 0$, $N \leq n_{i+1} - n_i \leq 2N-1$ and
  \begin{equation*}
    \Nt \nu_{n_{i+1}} \Nt_V + \alpha \sum_{k = n_i}^{n_{i+1}-1} \|
    \nu_k \|_{\varphi(V)} \leq \Nt \nu_0 \Nt_V,
    \qquad i \geq 0.
  \end{equation*}
  Summing this for all $i$, we get
  \begin{equation*}
    \alpha \sum_{k = 0}^{\infty} \| \nu_k \|_{\varphi(V)}
    \leq \Nt \nu_{0} \Nt_V,
    \qquad i \geq 0.
  \end{equation*}
  This shows that the sequence of probability measures $( S^k \mu_0 )_{k \geq 0}$
  is a Cauchy sequence in the norm $\| \cdot \|_{\varphi(V)}$, and
  hence converges to a certain probability measure $\mu^* \in  \P_{\varphi(V)}$
  which must satisfy $S \mu^* = \mu^*$ by construction. 
\end{proof}

\subsection{Uniqueness of equilibrium}
\label{sec:uniqueness}

Another consequence of Hypotheses \ref{hyp:local_coupling} and
\ref{hyp:op_Lyapunov_weak} is the uniqueness of equilibrium, that 
we present in two different frameworks.

\begin{cor}[Uniqueness of equilibrium]
  \label{cor:uniqueness}
  Let $S$ be a stochastic operator which satisfies the  Lyapunov condition (Hypothesis
    \ref{hyp:op_Lyapunov_weak}) and the   local coupling condition (Hypothesis \ref{hyp:local_coupling}) 
 of the existence Theorem \ref{thm:existence} for two couples
  of weight and sublinear functions $(V_1,\varphi_1)$ and $(V_2,\varphi_2)$ such that $\varphi_2(V_2) \ge V_1$.  
  Then $S$ has at most one equilibrium in $\P_{\varphi_2(V_2)}$.
\end{cor}

\begin{proof}[Proof of Corollary \ref{cor:uniqueness}]
  Let us consider two equilibria
  $\mu^*_1 ,\mu^*_2 \in \P_{\varphi_2(V_2)}$ and let us set
  $\nu := \mu^*_2 - \mu^*_1 \in \cN_{\varphi_2(V_2)} \subset
  \cN_{V_1}$.  From Lemma \ref{lem:nonexpansive} applied to
  $(V_1,\varphi_1)$ and because $S^k \nu = \nu$ for any $k \ge 0$, we
  get
$$
 \Nt \nu  \Nt_{V_1} + \alpha \sum_{k=0}^{n-1} \| \nu \|_{\varphi_1(V_1)}
    \leq \Nt \nu \Nt_{V_1},
    $$
for some equivalent norm $ \Nt \cdot  \Nt_{V_1}$, some integer $n \ge 1$ and some constant $\alpha > 0$. That implies $\| \nu \|_{\varphi_1(V_1)} = 0$, 
and thus $\mu^*_2= \mu^*_1$.
\end{proof}
\Black

We now consider the case when $S$ is a Feller-type stochastic operator. 

\begin{cor}[Uniqueness of equilibrium]
  \label{lem:uniqueness}
  Let $S$ be a Feller-type stochastic operator which satisfies the
  hypotheses of the existence Theorem \ref{thm:existence}.  Then $S$
  has a unique equilibrium $\mu \in \P_{\varphi(V)}$. 
\end{cor}

Before we prove that uniqueness result we will show that the weak
Lyapunov Hypothesis \ref{hyp:op_Lyapunov_weak} implies similar
inequalities for $\psi(V)$, where $\psi$ is a concave function:

\begin{lem}
  \label{lem:concave_psi_Lyapunov}
  Let $S$ be a Feller-type stochastic operator which satisfies the
  weak Lyapunov Hypothesis \ref{hyp:op_Lyapunov_weak} for $V$. From
  Remark~\ref{rem:FellerOp}, that is, $S = P^*$ and
  \begin{equation}
    \label{eq:weak_Lyapunov_dual}
    P V \leq V - \varsigma \varphi(V) + K.
  \end{equation}
  Then for any concave function $\psi \: [1,+ \infty) \to [1,+\infty)$, we have
  \begin{equation*}
    P \psi(V) \leq \psi(V) - \varsigma \psi'(V) \varphi(V) + K \psi'(V).
  \end{equation*}
\end{lem}

\begin{proof}[Proof of Lemma~\ref{lem:concave_psi_Lyapunov}]
  Notice that because $\psi$ is a  concave function, there holds
  $$
  \psi(v) = \inf_{\ell \in \cU_\psi} \ell(v), \quad \forall \, v \in \R, 
  $$
  where $\cU_\psi := \{\ell : \R \to \R, \,\, \ell(v) := a v + b, \,\, a, b \in \R, \,\, \ell \le \psi \}$. Using that $P$ is a positive operator, we deduce that 
  $$
  P\psi (V) \le P\ell (V) = \ell (PV), \quad \forall \, \ell \in \cU_\psi, 
  $$
  and then the Jensen's inequality
 \begin{equation}
    \label{eq:jensen2}
     P \psi(V) \leq \psi( P V).  
      \end{equation}
Using \eqref{eq:jensen2}, the nondecreasing property of $\psi$ (as emphasized in Remark~\ref{rem:weakLyap1}), \eqref{eq:weak_Lyapunov_dual} and 
the fact that $\psi$ is concave again, we get 
  \Black
  \begin{equation*}
    P \psi(V)
    \leq \psi( V - \varsigma \varphi(V) + K )
    \leq \psi(V) - \psi'(V)(\varsigma \varphi(V) - K),
  \end{equation*}
  which gives the inequality in the statement.
\end{proof}

\begin{proof}[Proof of Corollary \ref{lem:uniqueness}]
  The existence of an equilibrium is given by Theorem
  \ref{thm:existence}. Assume there are two equilibria
  $\mu^*_1, \mu^*_2 \in \P_{\varphi(V)}$, and call $\nu := \mu^*_1 - \mu^*_2$, so that in particular 
  $\nu \in \cN_{\varphi(V)}$ and  $S \nu = \nu$. 
   Similarly  as in the proof of Corollary \ref{cor:uniqueness}, we would like to use the weak Lyapunov 
   condition from Hypothesis  \ref{hyp:op_Lyapunov_weak} in order to get 
   \begin{equation}
    \label{eq:8}
       \varsigma \|\nu \|_{\varphi(V)} \leq K \|\nu\|,
   \end{equation}
   but this is not allowed because
   $\nu$ is not necessarily in $\cN_V$ and we cannot justify cancelling
   the term $\|\mu\|_V$ on both sides. Hence we carry out an
   approximation procedure in order to deduce \eqref{eq:8}. 
   Since $S$ is of Feller-type,
   $S = P^*$ for a Markov-Feller operator $P$.
   Take  $\psi \: [1,+\infty) \to [1,+\infty)$ a bounded concave function such that $ \psi'(v) \le 1$ for all $v \geq 1$, 
   so that 
  \begin{equation*}
    P \psi(V) \leq \psi(V) - \varsigma \psi'(V) \varphi(V) + K,
  \end{equation*}
  from Lemma \ref{lem:concave_psi_Lyapunov}. After integration and by duality, for any $0 \le \mu \in \M_{\varphi(V)}$, we have 
  \begin{equation*}
    \int \psi(V) S \mu
    \leq
    \int \psi(V) \mu  - \varsigma \int \psi'(V) \varphi(V) \mu
    + K \int \mu.
  \end{equation*} \Black
   Applying this to $\mu := |\nu| = |\mu^*_1 - \mu^*_2| \in \M_{\varphi(V)}$, we get
  \begin{equation*}
    \int \psi(V) |S \nu|
    \leq
    \int \psi(V) S |\nu|
    \leq
    \int \psi(V) |\nu|  - \varsigma \int \psi'(V) \varphi(V) |\nu|
    + K \int |\nu|,
  \end{equation*}
  and since $S \nu = \nu$, we deduce
   \begin{equation}
    \label{eq:8bis}
     \varsigma \int \psi'(V) \varphi(V) |\nu|
    \leq
    K \int |\nu|.
  \end{equation}
  Taking for example
  $\psi_n(v) := n \arctan \big( \frac{\pi}{2} + v/n \big)$,  so that $\psi_n'(v) \nearrow 1$ as $n\to\infty$, and passing
  to the limit as $n \to +\infty$ in \eqref{eq:8bis}, the
  dominated convergence theorem shows that \eqref{eq:8} holds true. 
Since \eqref{eq:8} holds, the iterated
  coupling condition (Hypothesis 2 in  Lemma~\ref{lem:nonexpansive}) gives that
  \begin{equation*}
    \| \nu \| = \| S^N \nu \| \leq \gamma_H \|\nu\|,
  \end{equation*}
  which implies $\|\nu\| = 0$ and hence $\mu^*_1 = \mu^*_2$.
\end{proof}

\subsection{Subgeometric decay rates}
\label{sec:rate}

For a nonempty interval $I \subseteq \R_+$ and a  function $\xi \: I \to \R$, we recall
that the associated Legendre transform $\xi^* : \R \to \R \cup \{+\infty\}$ defined by 
\begin{equation*}
  \xi^* (u) := \sup_{\lambda \in I} (\lambda u - \xi(\lambda)),
\end{equation*}
is an increasing
and convex function, and in particular it is continuous on the interior of the interval $D(\xi^*) := \{ u \in \R; \,   \xi^* (u) < + \infty\}$.
We also define the closely related transform
\begin{equation*}
  \xi_* (u) := \sup_{\lambda \in I} (\xi(\lambda) - \lambda u)
  = (-\xi)^*(-u),
\end{equation*}
also defined (and possibly $+\infty$) at all $u \in \R$.

\medskip
Our main theorem in the subgeometric case is the following, which
involves two different weight functions $V_1$ and $V_2$:

\begin{thm}[Subgeometric Harris, interpolated version]
  \label{theo:HarrisSubgeo1}
  Consider a stochastic operator $S$ such that:
  \begin{enumerate}
  \item $S$ satisfies a weak Lyapunov condition (Hypothesis
    \ref{hyp:op_Lyapunov_weak}) for two couples of weight and sublinear functions 
    $(V_1,\varphi_1)$, $(V_2,\varphi_2)$ and constants $K_1, \varsigma_1, K_2, \varsigma_2$,
    respectively, and such that $V_1 \leq V_2$.
  \item There exists an integer $N \geq 1$ such that $S^N$ satisfies a
    local coupling condition (Hypothesis \ref{hyp:local_coupling}) for
    both $\varphi_1(V_1)$ and $\varphi_2(V_2)$, with constants
    $A_1 > K_1 / \varsigma_1$, $A_2 > K_2 / \varsigma_2$  and same constant $\gamma_H$.
    
  \item The following interpolation condition holds: there is  function $\xi \:  \R_+ \to \R_+$ which is 
    increasing and satisfies
    $\xi(\lambda) / \lambda \to 0 \ \hbox{as} \ \lambda \to 0$ and such
    that
    \begin{equation}
      \label{eq:InterpolationV1V0V3}
      \lambda V_1 \le \varphi_1(V_1) + \xi(\lambda) V_2,
      \qquad \text{for all $\lambda > 0$.}  
    \end{equation}
  \end{enumerate}
  Then there exist constructive constants $C > 0$ and $0 < r < 1$
  (depending only on $\xi$, $K_i$, $\varsigma_i$, $A_i$ for $i = 1,2$,
  and on $\gamma_H$) such that
  \begin{equation}
    \label{eq:theo:HarrisSubgeo1:Estim1}
    \| S^n \nu \|_{V_1}
    \leq
    C \Theta \left( r n \right)
    \| \nu \|_{V_2},
    \qquad \text{for all $n \ge 1$,}
  \end{equation}
  and
  \begin{equation}
    \label{eq:theo:HarrisSubgeo1:Estim2}
    \| S^n \nu \|
    \leq
    \frac{C}{n}
    \Theta \left( r n \right) \| \nu \|_{V_2},
    \qquad \text{for all $n \ge 1$,}
  \end{equation}
  for any $\nu \in \cN_{V_2}$,
  where
  \begin{equation*}
    \Theta(t) := F^{-1}(t),
    \qquad
    F(\lambda) := \int_\lambda^1 \frac{1}{\xi^*(s)} \d s.
  \end{equation*}
\end{thm}

\begin{rem} 
(1) -  Under the assumption of Theorem~\ref{theo:HarrisSubgeo1} and when $V_1 \le \varphi_2(V_2)$ or $S$ is of Feller-type, we get the 
 existence of an equilibrium (Theorem~\ref{thm:existence}), its uniqueness (Corollary~\ref{cor:uniqueness} or Corollary~\ref{lem:uniqueness}) and a decay rate of convergence to zero (Theorem~\ref{theo:HarrisSubgeo1}). 
 
 \smallskip
 (2) - Let us emphasize that, in contrast with Theorem~\ref{thm:Doeblin} and Theorem~\ref{thm:Harris}, in principle there is no reason that $\mu^*$ belongs to $\P_V$, and thus, we cannot apply Theorem~\ref{theo:HarrisSubgeo1} to $\mu-\mu^*$
 and deduce $S^n \mu \to \mu^*$ as $n \to \infty$  (with or even without rate!). We will come back on that issue in Remark~\ref{rem:subgeo-stabmu*} below.  
\end{rem}

In the rest of this section we prove Theorem
\ref{theo:HarrisSubgeo1}. We start with a finite difference inequality
which is at the basis of the estimates we carry out in the proof:

\begin{lem}
  \label{lem:discrete_ode}
  Let $(u_n)_{n \geq 0}$ be a nonnegative sequence which
  satisfies
  \begin{equation}
    \label{eq:lem_ode}
    u_{n+1} - u_n \leq - g(u_n)
    \qquad \text{for all integers $n \geq 0$,}
  \end{equation}
  for some continuous, increasing function
  $g \: (0,u_0] \to (0,+\infty)$ such that $v \mapsto 1/g(v)$ is not
  integrable on $(0,u_0)$. Then
  \begin{equation*}
    u_n \leq {H^{-1}(n)}
    \quad \text{for all integers $n \geq 0$,}
  \end{equation*}
  where
  \begin{equation*}
    H(u) := \int_u^{u_0} \frac{1}{g(v)} \d v
    \qquad \text{for $u \in (0,u_0]$.}
  \end{equation*}
\end{lem}

\begin{proof}[Proof of lemma~\ref{lem:discrete_ode}]
  Let $u = u(t)$ be the solution for $t \geq 0$ to the ordinary
  differential equation
  \begin{equation*}
    u'(t) = -g(u(t)), \qquad u(0)=u_0,
  \end{equation*}
  which is precisely $u(t) = H^{-1}(t)$. We prove by
  induction that $u_n \leq u(n)$ for all $n \geq 0$. It is indeed true for
  $n=0$. If we assume $u(n) \geq  u_n$ for some $n \geq 0$, then
  \begin{equation*}
    u(n+1) \geq \tilde{u}(n+1),
  \end{equation*}
  where $\tilde{u}$ is the solution to
  \begin{equation*}
    \tilde{u}' = -g(\tilde{u}), \qquad \tilde{u}(n) = u_n.
  \end{equation*}
Since $\tilde{u}$ is decreasing and $g$ is increasing, we then have 
  \begin{equation*}
    \tilde{u}(n+1) = \tilde{u}(n) - \int_n^{n+1} g( \tilde{u}(t)) \d t
    \geq \tilde{u}(n) - g(\tilde{u}(n))
    = u_n - g(u_n) \geq u_{n+1},
  \end{equation*}
  which shows $u_{n+1} \leq u(n+1)$.
\end{proof}

\begin{lem}[Difference inequality]
  \label{lem:dif_ineq}
  Take $M > 0$, $0 < \delta \leq +\infty$, and
  $\zeta \: (0, \delta) \to (0,+\infty)$ a nonnegative function
  satisfying $\lim_{\lambda \to 0} \zeta(\lambda) / \lambda = 0$. If a
  sequence $(u_n)_{n \geq 0}$ of nonnegative numbers satisfies
  $u_0 \leq M$ and
  \begin{equation}
    \label{eq:difineq}
    u_{n+1} \leq (1 - \lambda) u_n + M \zeta(\lambda)
    \qquad \text{for all $n \geq 0$ and all $\lambda \in (0,\delta)$,}
  \end{equation}
  then
  \begin{equation*}
    u_n \leq M F^{-1}(n)
    \quad \text{for all integers $n \geq 0$,}
  \end{equation*}
  where
  \begin{equation*}
    F(u) := \int_u^1 \frac{1}{\zeta^*(v)} \d v
    \qquad \text{for $u \in (0,1]$,}
  \end{equation*}
  and $\zeta^*$ denotes the Legendre transform of $\zeta$.
\end{lem}

\begin{proof}[Proof of Lemma~\ref{lem:dif_ineq}]
  Call $v_n := u_n / M$ for $n \geq 0$. Minimising \eqref{eq:difineq} 
  in $\lambda$ we obtain
  \begin{equation*}
    v_{n+1} - v_n \leq -\zeta^* (v_n)
    \qquad \text{for $n \geq 0$}.
  \end{equation*}
  In particular $\zeta^*(v_n)$ must always be finite on $(0,v_0]$,
  since the sequence $(v_n)$ is assumed to be a sequence of
  nonnegative numbers, and $\zeta^*$ is a nondecreasing
  function. Notice also that $\zeta^*$ is continuous since it is
  convex, and that the condition
  $\lim_{\lambda \to 0} \zeta(\lambda)/\lambda = 0$ ensures
  $\zeta^*(v) > 0$ for all $v \in (0,v_0]$. Then Lemma \ref{lem:discrete_ode}
  with $g \equiv \zeta^*$ gives
  \begin{equation}
    \label{eq:11}
    v_n \leq \widetilde{F}^{-1}(n),
  \end{equation}
  where
  \begin{equation*}
    \widetilde{F}(v) := \int_v^{v_0} \frac{1}{\zeta^*(s)} \d s
    \qquad \text{for $v \in (0,v_0]$.}
  \end{equation*}
  Since $v_0 = u_0 / M \leq 1$ by assumption, this shows
  \begin{equation*}
    \widetilde{F}(v) \leq \int_v^{1} \frac{1}{\zeta^*(s)} \d s =: F(v)
    \qquad \text{for $v \in (0,v_0]$,}
  \end{equation*}
  where we understand $1/\zeta^*(s) = 0$ is $\zeta^*(s) =
  +\infty$. Hence \eqref{eq:11} gives
  \begin{equation*}
    u_n \leq M F^{-1}(n)
  \end{equation*}
  for all $n \geq 0$, as required.
\end{proof}

\smallskip

We also need a technical lemma which will be used to simplify the
bounds in our main results.

\begin{lem}
  \label{lem:theta-bound} Let $g \: (0,+\infty) \to (0,+\infty)$ be a
  positive, nondecreasing function with $\lim_{s \to 0} g(s)/s = 0$,
  and define
  \begin{equation*}
    F(\lambda) := \int_\lambda^1 \frac{1}{g(s)} \d s
    \qquad \text{for $0 < \lambda \leq 1$.}
  \end{equation*}
  (We notice that $F^{-1} \: [0,+\infty) \to (0,1]$ is well defined, continuous and strictly decreasing, 
  since $F$ is strictly decreasing and
  $\lim_{\lambda \to 0} F(\lambda) = +\infty$.) Then for any $k > 0$,
  there exists a constant $C > 1$ which depends only on $k$ and $g$,
  such that
  \begin{equation*}
    F^{-1}(t - k) \leq C F^{-1}(t)
    \qquad \text{for all $t \geq k$.}
  \end{equation*}
\end{lem}

\begin{proof}[Proof of Lemma~\ref{lem:theta-bound}]
  We first notice that for any $C > 1$ and $0 < \lambda < 1/C$,
  \begin{equation*}
    F(C \lambda) = \int_{C \lambda}^1 \frac{1}{g(s)} \d s
    = F(\lambda) - \int_{\lambda}^{C \lambda} \frac{1}{g(s)} \d s
    \leq F(\lambda) - (C-1) \frac{\lambda}{g(C \lambda)}.
  \end{equation*}
  Using that $\lim_{\lambda \to 0} \lambda / g(C \lambda) = +\infty$,
  we may take $\lambda_0 < 1/C$ small enough so that
  \begin{equation*}
    (C-1) \frac{\lambda}{g(C \lambda)} \geq k
    \qquad
    \text{for all $0 < \lambda < \lambda_0$},
  \end{equation*}
  so
  \begin{equation*}
    F(C \lambda)
    \leq F(\lambda) - k
    \qquad
    \text{for all $0 < \lambda < \lambda_0$}.
  \end{equation*}
  so setting $\lambda := F^{-1}(t)$ for some $t > F(\lambda_0)$ we get
  \begin{equation*}
    F(C F^{-1}(t)) \leq t - k
    \qquad
    \text{for all $t > F(\lambda_0)$},
  \end{equation*}
  which after applying $F^{-1}$ gives the inequality in the lemma
  whenever $t > F(\lambda_0)$. The inequality is also clearly true,
  with some other constant $C$, for all $t \in [k, F(\lambda_0)]$,
  since this is a compact interval.
\end{proof}

\medskip We are now ready to give the proof of
Theorem~\ref{theo:HarrisSubgeo1}:

\begin{proof}[Proof of Theorem~\ref{theo:HarrisSubgeo1}]
  Let $\Nt \cdot \Nt_{V_1}$ and $\Nt \cdot \Nt_{V_2}$ denote the norms
  from Lemma \ref{lem:nonexpansive}, equivalent to $\| \cdot \|_{V_1}$
  and $\|\cdot\|_{V_2}$ respectively, defined by
  \begin{gather*}
    \Nt \mu \Nt_{V_1}
    := \| \mu \| + \beta_1 \| \mu \|_{V_1},
    \quad \forall \, \mu \in \M_{V_1},
    \\
    \Nt \mu \Nt_{V_2}
    := \| \mu \| + \beta_2 \| \mu \|_{V_2},
    \quad \forall \, \mu \in \M_{V_2},
  \end{gather*}
  with
  \begin{gather*}
    \beta_1 := (1-\gamma_H)/(K_1 N),
    \qquad
    \beta_2 := (1-\gamma_H)/(K_2 N).
  \end{gather*}
  We also take
  \begin{equation*}
    \alpha :=
    \min\{ \beta_1 (\varsigma_1 - K_1/A_1),\
    \beta_2 (\varsigma_2 - K_2/A_2)   \} > 0.
  \end{equation*}
  Take $\nu \in \mathcal{N}_{V_2}$, and denote $\nu_k := S^k \nu$ for
  integer $k \geq 0$.

  \smallskip \noindent \emph{Step 1. Uniform bound on the $V_2$ norm.} 
  Using Lemma \ref{lem:nonexpansive} for $V_2$, we can recursively define an increasing sequence of
  integers $(n_i)_{i \geq 0}$ such that $n_0 = 0$,
  $N \leq n_{i+1} - n_i \leq 2N-1$ for all $i$ and which satisfy
  \begin{equation}
    \label{eq:V2bound0}
    \Nt \nu_{n_{i+1}} \Nt_{V_2}
    \leq \Nt \nu_{n_i} \Nt_{V_2},
    \qquad i \geq 1.
  \end{equation}
On the other hand, using the weak Lyapunov condition (Hypothesis
    \ref{hyp:op_Lyapunov_weak}) and the total variation  non-expansive property 
\eqref{eq:Doeblin-contrac1}, we have 
$$
  \Nt \nu_{k+1} \Nt_{V_2} \leq  \Nt \nu_k \Nt_{V_2} +  \beta_2 K_2   \| \nu_k \| \le (1 + \beta_2K_2 ) \Nt \nu_k \Nt_{V_2}, 
$$ 
and thus 
  \begin{equation}
    \label{eq:V2bound}
    \Nt \nu_k \Nt_{V_2} \leq C_2 \Nt \nu \Nt_{V_2}, 
    \qquad \text{for all $k \geq 0$,}
  \end{equation}
  with $C_2 := (1 + \beta_2K_2)^N$.
  
  \smallskip \noindent \emph{Step 2. Decay along a subsequence.}
  Using again Lemma \ref{lem:nonexpansive}, now for $V_1$, we can
  recursively define a (possibly different) increasing sequence of
  integers $(n_i)_{i \geq 0}$ such that $n_0 = 0$,
  $N \leq n_{i+1} - n_i \leq 2N-1$ for all $i$ and which satisfy
  \begin{equation}
    \label{eq:6}
    \Nt \nu_{n_{i+1}} \Nt_{V_1}
    + \alpha \sum_{k= n_i}^{n_{i+1}-1}
    \| \nu_{k} \|_{\varphi_1(V_1)}
    \leq \Nt \nu_{n_i} \Nt_{V_1},
    \qquad i \geq 1,
  \end{equation}
  and in particular
  \begin{equation}
    \label{eq:4}
      \Nt \nu_{n_{i+1}} \Nt_{V_1}
    + \alpha \| \nu_{n_{i}} \|_{\varphi_1(V_1)}
    \leq \Nt \nu_{n_i} \Nt_{V_1},
    \qquad i \geq 1,
  \end{equation}
  where we have ignored all terms in the sum except for the first
  one. From \eqref{eq:4} and the interpolation condition
  \eqref{eq:InterpolationV1V0V3} we deduce, for any $\lambda > 0$,
  $$
  \Nt \nu_{n_{i+1}} \Nt_{V_1}
  + \lambda \alpha \| \nu_{n_{i}} \|_{V_1}
  \le
  \Nt \nu_{n_{i}} \Nt_{V_1}
  + \xi(\lambda) \alpha  \| \nu_{n_{i}} \|_{V_2}.
  $$
  We now use that the norms $\|\cdot\|_{V_1}$ and $\|\cdot\|_{V_2}$
  are equivalent, respectively, to $\Nt \cdot \Nt_{V_1}$ and
  $\Nt \cdot \Nt_{V_2}$: from the definition of $\Nt \cdot \Nt_{V_1}$,
  \begin{equation*}
    \lambda \alpha \| \nu_{n_{i}} \|_{V_1}
    \geq
    \lambda \kappa \Nt \nu_{n_{i}} \Nt_{V_1}
    \qquad \text{with $\kappa := \frac{\alpha}{1+\beta_1}$,}
  \end{equation*}
  for any $\lambda > 0$. Also, from the definition of the
  $\Nt\cdot \Nt_{V_2}$ norm and \eqref{eq:V2bound},
  $$
  \| \nu_{n_{i}} \|_{V_2}
  \leq \frac{1}{\beta_2} \Nt \nu_{n_{i}} \Nt_{V_2}
  \leq \frac{C_2}{\beta_2} \Nt \nu \Nt_{V_2}.
  $$
  The three previous estimates together imply
  \begin{equation*}
    \Nt \nu_{n_{i+1}} \Nt_{V_1}
    \leq (1 - \lambda \kappa) \Nt \nu_{n_{i}} \Nt_{V_1} +
    \frac{\xi(\lambda) C_2 \alpha}{\beta_2} \Nt \nu \Nt_{V_2},
  \end{equation*}
  for any $i\ge 0$ and for any $\lambda > 0$. Using that
  $V_1 \leq V_2$, so $\Nt \nu \Nt_{V_1} \leq \Nt \nu \Nt_{V_2}$, Lemma
  \ref{lem:dif_ineq} with $\zeta(\lambda) := \xi(\lambda/\kappa)$ and
  $M := \Nt \nu \Nt_{V_2} \max\{1, C_2 \alpha/ \beta_2\}$ then implies
  \begin{equation}
    \label{eq:decay-ki}
    \Nt \nu_{n_{i}}  \Nt_{V_1}
    \leq
    \frac{m}{\kappa} \Nt   \nu  \Nt_{V_2} \Theta(\kappa i)
    \qquad \text{for all $i \geq 1$},
  \end{equation}
  where $m := \max\{1, C_2 \alpha/ \beta_2\}$ and $\Theta$ is the decay
  rate function defined in the statement. 
  Notice that we have used that 
  $\zeta^*(s) = \xi^*(\kappa s)$ for all $s \in \R$, and that
  \begin{equation*}
    \int_{\lambda}^1 \frac{1}{\zeta^*(s)} \d s
    = \int_{\lambda}^1 \frac{1}{\xi^*(\kappa s)} \d s
    = \frac{1}{\kappa} \int_{\kappa \lambda}^\kappa \frac{1}{\xi^*(s)} \d s
    \leq \frac{1}{\kappa} F(\kappa \lambda),
  \end{equation*}
  so the decay rate in Lemma \ref{lem:dif_ineq} is bounded by the one
  given in \eqref{eq:decay-ki}.

  \smallskip \noindent \emph{Step 3. Decay along the full sequence.}
  Now we have proved this decay rate along the sequence
  $(n_i)_{i \geq 0}$.   In order to extend this to all indices $k$,
  we observe that proceeding exactly as in the proof of \eqref{eq:V2bound}, 
  we get
   \begin{equation}
    \label{eq:5}
    \Nt S^i \nu_k \Nt_{V_1} \leq C^i_1 \Nt \nu_k \Nt_{V_1}, 
    \qquad \text{for all $k , i\geq 0$,}
  \end{equation}
  with $C_1 := 1 + \beta_1K_1$. \Black
q For any $k \geq 0$, choose
  $j \geq 0$ such that $n_j \leq k < n_{j+1}$. Due to the spacing of
  the terms $n_i$, it must hold that
  \begin{equation}
    \label{eq:7}
    \lfloor k/(2N - 1) \rfloor \leq j \leq \lfloor k/N \rfloor.
  \end{equation}
  Writing $k = n_j + i$ for some $0 \leq i \leq 2N-2$, we have, using
  \eqref{eq:decay-ki}, \eqref{eq:5} and \eqref{eq:7},
  \begin{multline*}
    \Nt \nu_k \Nt_{V_1} = \Nt S^i \nu_{n_j} \Nt_{V_1}
    \leq C_1^i \Nt \nu_{n_j} \Nt_{V_1}
    \\
    \leq
    C_1^{2N-2} \frac{m}{\kappa} \Theta(\kappa j)
    \Nt \nu  \Nt_{V_2}
    \leq
    C \Theta \left( \kappa \left\lfloor
        \frac{k}{2N-1}
      \right\rfloor \right)
    \Nt \nu  \Nt_{V_2},
  \end{multline*}
  where the constant $C$ is given by
  \begin{equation*}
    C := C_1 ^{2N-2} \frac{m}{\kappa}
    = C_1 ^{2N-2} \frac{1 }{\kappa } \max\{1, C \alpha/ \beta_2\}.
  \end{equation*}

  \smallskip \noindent \emph{Step 4. Simplification of the decay
    rate.}
  We notice that
  \begin{equation*}
    \left\lfloor
      \frac{k}{2N-1}
    \right\rfloor
    \geq \frac{k}{2N-1} - 1
    \qquad
    \text{for all $k \geq 0$,}
  \end{equation*}
  so we may use Lemma \ref{lem:theta-bound} to obtain
  \begin{equation*}
    \Theta \left( \kappa \left\lfloor
        \frac{k}{2N-1}
      \right\rfloor \right)
    \leq
    \Theta \left ( \frac{\kappa k}{2N-1} - \kappa \right)
    \leq C \Theta \left ( \frac{\kappa k}{2N-1} \right)
  \end{equation*}
  for all $k \geq 2N - 1$ and some constant $C > 0$. Since the inequality
  \begin{equation*}
    \Theta \left( \kappa \left\lfloor
        \frac{k}{2N-1}
      \right\rfloor \right)
    \leq C \Theta \left ( \frac{\kappa k}{2N-1} \right)
  \end{equation*}
  is clearly also true for some (other) $C \geq 1$ and the finite set
  of integers $0 \leq k \leq 2N-1$, we obtain the form of the decay
  rate given in the statement.

  \smallskip \noindent \emph{Step 5. Decay in total variation norm.}
  In order to deduce the second estimate
  \eqref{eq:theo:HarrisSubgeo1:Estim2}, we come back to the first
  inequality in \eqref{eq:6} that we iterate and sum up in order to
  obtain, for any $0 \leq j < i$,
  $$
  \Nt \nu_{n_i} \Nt_{V_1}
  + \alpha \sum_{k=n_j}^{n_i-1} \| \nu_k \|_{\varphi_1(V_1)}  
  \le  \Nt \nu_{n_j}  \Nt_{V_1}.
  $$
  Together with the non expansion inequality
  $$
  \| \nu_{n_i} \| \le \| \nu_k \| \le
  \| \nu_k \|_{\varphi_1(V_1)}, \quad \forall \, k \leq n_i,
  $$
  and the decay proved in \eqref{eq:decay-ki}, we deduce
  $$
  \bigl( n_i - n_j \bigr) \alpha \| \nu_{n_i} \|
  \leq M \Nt \nu  \Nt_{V_2} \Theta(j).
  $$
  Choosing $j = \lfloor i/2 \rfloor$ and using that $n_i - n_j \geq N
  (j-i)$,
  $$
  \| \nu_{n_i} \|
  \leq \frac{2 M}{\alpha i} \Nt \nu  \Nt_{V_2} \Theta(\lfloor i/2 \rfloor).
  $$
  Carrying out a similar argument as above to extend this to all
  indices $k$, we obtain
  $$
  \| \nu_{k} \|
  \leq \frac{C_3}{k}
  \Theta\Big(
  \Big\lfloor \frac{k}{4N-2} \Big\rfloor
  \Big) \Nt \nu  \Nt_{V_2}, 
  $$
  for all $k \geq 1$, for some other constant $C_3 >
  0$.
  A similar reasoning as in the previous step gives the simpler form
  of the decay rate given in the statement.
\end{proof}


\subsection{Subgeometric decay rates for Feller type stochastic operator}
\label{sec:alternative}

As a consequence of Theorem \ref{theo:HarrisSubgeo1}, we can prove the
following theorem adapted to Feller type stochastic operator and which is closer to the continuous-time framework
developed in
\citet{GDF2009},  see also \citet{Hairer2016notes}.

\begin{thm}[Discrete subgeometric Harris theorem for Feller type stochastic operator]
  \label{thm:subgeometric_Harris_discrete}
  Consider a stochastic operator $S$ of Feller type such that
  \begin{enumerate}
  \item $S$ satisfies a weak Lyapunov condition (Hypothesis
    \ref{hyp:op_Lyapunov_weak}) with functions $V$, $\varphi$ and
    constants $\varsigma$, $K$.
  \item For some $N \geq 1$, $S^N$ satisfies a Harris condition
    (Hypothesis \ref{hyp:Harris}) on the set
    $\mathcal{C} := \{ \varphi(V) \leq 2 R \}$ for some
    $R > 2 K / \varsigma$.
  \end{enumerate}
  Then there exists a unique equilibrium $\mu^* \in \P_{\varphi(V)}$.
  Moreover, for any strictly concave function
  $\psi \: [1,+\infty) \to [1,+\infty)$ with $\psi(1)= \psi'(1) = 1$,
  $\lim_{v \to +\infty} \psi(v) = +\infty$, and such that
  $v \mapsto \psi' (v) \varphi(v)$ is nondecreasing and satisfies
  $\psi' (v) \varphi(v) > R$ whenever $\varphi(v) > 2R$, we have
  \begin{equation}
    \label{eq:HarrisSubgeo2}
    \| S^n \nu \| \leq
    \frac{C}{n} \Theta_{\psi} ( r n) \| \nu \|_{V}
    \qquad \text{for all $n \ge 1$},
  \end{equation}
  for any $\nu \in \cN_{V}$. Here $C \geq 1$ and $0 < r < 1$ are
  constructive constants and the
  decay rate $\Theta_\psi$ is given by
  \begin{equation*}
    {\Theta}_\psi(n) := F_\psi^{-1}( n ),
  \end{equation*}
  where
  \begin{gather*}
    F_\psi(v) := \int_v^1 \frac{1}{h(u)} \d u,
    \qquad
    h(u) := g \circ f^{-1} (u),
    \\
    f(v) := \frac{\psi(v)}{v},
    \qquad
    g(v) := \frac{\psi'(v) \varphi(v)}{v}.
  \end{gather*}
\end{thm}

\begin{rem}\label{rem:subgeo-stabmu*} Provided that
  $\varphi'(v) \varphi(v) \ge \Phi(\varphi(v))$ for any $v \ge 1$ for
  some concave function $\Phi : [1,\infty) \to [1,\infty)$ and $S^N$
  satisfies a Harris condition (Hypothesis \ref{hyp:Harris}) on the
  set $\{ \varphi(V) \leq 2 R \}$ for any $R > 2K/\varsigma$, the
  techniques developed here make possible to establish that $\mu^*$ is
  asymptotically stable:
  there exists a decay rate function $\widetilde\Theta$ such that for
  any $\mu \in \P_{\varphi(V)}$ there holds
  \beqn
  \label{eq:EstimStabSubGeo}
  \| S^n \mu - \mu^* \| \le
  \widetilde\Theta (n) \| \mu - \mu^* \|_{\varphi(V)},
  \quad \forall
  \, n \ge 1.
  \eeqn
  Defining indeed the weight function
  $W:=\varphi(V)$, Lemma~\ref{lem:concave_psi_Lyapunov} implies that
  $$
  PW + \varsigma\Phi(W) \le W + K, 
  $$
  which is nothing but saying that $S$ satisfies a weak Lyapunov
  condition (Hypothesis \ref{hyp:op_Lyapunov_weak} and
  Remark~\ref{rem:FellerOp}) with functions $W$, $\Phi$ and constants
  $\varsigma$, $K$.  Because of the above strong Harris condition,
  Theorem~\ref{thm:subgeometric_Harris_discrete} holds with $W$ for a
  family of rate functions $\widetilde\Theta$. For any on them, we may
  apply Theorem~\ref{thm:subgeometric_Harris_discrete} with
  $\nu := \mu - \mu^* \in \P_W$ and deduce \eqref{eq:EstimStabSubGeo}.
\end{rem}

\begin{proof}[Proof of Theorem \ref{thm:subgeometric_Harris_discrete}]
  First, we notice that $f$ is invertible since $v \mapsto \psi(v)/v$
  is strictly decreasing, as can be seen from
  \begin{equation*}
    \frac{\mathrm{d}}{\mathrm{d}v}
    \frac{\psi(v)}{v}
    = \frac{v \psi'(v) - \psi(v)}{v^2} < 0
    \qquad \text{for $v > 1$,}
  \end{equation*}
  since $v \psi'(v) < \psi(v)$ for $v > 1$ due to the strict concavity
  of $\psi$ and the fact that $\psi'(1)=1$.

  In order to show the result we use Theorem \ref{theo:HarrisSubgeo1}
  with
  \begin{equation*}
    V_2 := V,
    \quad
    V_1 := \psi(V).
  \end{equation*}
  Let us check the assumptions of Theorem
  \ref{theo:HarrisSubgeo1}. First, the weak Lyapunov conditions
  \eqref{eq:weak_Lyapunov_implicit} and \eqref{eq:weak_Lyapunov_dual}
  are satisfied for $V_2 = V$ by assumption. In order to see that a
  weak Lyapunov condition holds also for $V_1 = \psi(V)$, use that
  $\psi$ is concave to write, with Jensen's inequality
  \eqref{eq:jensen2} and \eqref{eq:weak_Lyapunov_dual},
  \begin{multline*}
    P \psi(V)
    \leq \psi( V - \varsigma \varphi(V) + K )
    \leq \psi(V) - \varsigma \psi'(V)\varphi(V) + \psi'(V) K
    \\
    \leq \psi(V) - \varsigma \psi'(V)\varphi(V) + K
    = \psi(V) - \varsigma \varphi_1(\psi(V)) + K,
  \end{multline*}
  where we make the choice
  \begin{equation*}
    \varphi_1(\psi(v)) := \psi'(v)\varphi(v).
  \end{equation*}
  Notice that $\varphi_1$ is a nondecreasing function with
  $\varphi_1(w) \geq \varphi_1(1) = 1$.  Observing that
  $$
  P\psi(V \wedge n) \le P\psi(V) \le \psi(V) - \varsigma \varphi_1(\psi(V)) + K, 
  $$
 by   duality for any $0 \le \mu \in \M_{\psi(V)}$, we have 
  $$
  \| S \mu \|_{\psi(V \wedge n)} =   \int  \mu P(V \wedge n) \le \int  \mu  (\psi(V) - \varsigma \varphi_1(\psi(V)) + K).
  $$
  By Beppo Levi theorem, we may pass to the limit in the above inequality and we obtain that the weak Lyapunov condition \eqref{eq:weak_Lyapunov_implicit} also holds for $V_1 = V$.
Notice that both weak Lyapunov
  conditions for $V_1$ and $V_2$ hold with the same constants $\varsigma$
  and $K$.

  The Harris condition for $S^N$ is satisfied on the set
  \begin{equation*}
    \tilde{\mathcal{C}} = \{x \in \Omega \mid
    \varphi_1(\psi(v)) \leq R \}, 
  \end{equation*}
  since by hypothesis the condition
  $\varphi_1(\psi(v)) = \psi'(v) \varphi(v) \leq R$ implies
  $\varphi(v) \leq 2 R$, so
  $\tilde{\mathcal{C}} \subseteq \mathcal{C}$. Lemma
  \ref{lem:Harris&coupling} shows that $S^N$ satisfies the local
  coupling condition (Hypothesis \ref{hyp:local_coupling}) for
  $\varphi_1(\psi(V))$ (and hence for $\varphi(V)$, which is larger),
  both with any constant $A < R/2$. Since $R > 2K / \varsigma$ we may
  take $A > K/\varsigma$,  and the hypotheses of Theorem
  \ref{theo:HarrisSubgeo1} are met. 
  In order to express the conclusion of Theorem~\ref{theo:HarrisSubgeo1}, 
  we observe that the interpolation function $\xi$ in \eqref{eq:InterpolationV1V0V3}
  can be written more explicitly in the present case where 
$
    V_1 = \psi(V_2)
$
  for some function $\psi \: [1, +\infty) \to [1,+\infty)$ with
  $\psi(v)/v$ strictly decreasing. Indeed, in that case, the interpolation is
  equivalent to
  \begin{equation}
    \label{eq:9}
    \xi(\lambda) \geq \lambda f(v) - g(v)
    \qquad \text{for all $v \geq 1$,}
  \end{equation}
  where
  \begin{equation*}
    f(v) := \psi(v) / v,
    \qquad
    g(v) := \varphi_1(\psi(v)) / v.
  \end{equation*}
  Substituting $v = f^{-1}(z)$ in \eqref{eq:9}, $\xi(\lambda)$ must satisfy
  \begin{equation*}
    \xi(\lambda) \geq \lambda z - g(f^{-1}(z))
    \qquad \text{for all $0 < z \leq 1$,}
  \end{equation*}
  so we can choose
  \begin{equation*}
    \xi(\lambda) := h^*(\lambda),
    \quad \text{where $h \: (0,1] \to \R$ is given by $h(z) := g (f^{-1}(z))$.}
  \end{equation*}
  Thus we have $\xi^* = h$ and we obtain $F = F_\psi$ in the conclusion of  
  Theorem~\ref{theo:HarrisSubgeo1}.
\end{proof}

There remains the question of choosing the function $\psi$ which gives
an optimal decay rate $\widetilde{\Theta}_\psi$. There are two
``extreme'' choices for $\psi$: one can take $\psi$ asymptotically
like
\begin{equation*}
  H(u) := \int_1^u \frac{1}{\varphi(v)} \d v
\end{equation*}
(so that $\psi'(v) \varphi(v)$ behaves like a constant as $v \to
+\infty$); or one can take $\psi(v)$ almost equal to $v$ (but still
strictly concave). These are both useful in different cases, as we
show now in examples:

\paragraph{Polynomial decay. } 
Let us take
\begin{equation*}
  \psi(u) := 1 + \int_1^u \frac{m(v)}{\varphi(v)} \d v,
  \qquad u \geq 1,
\end{equation*}
for some continuous, nondecreasing $m \: [1,+\infty) \to [1,+\infty)$ such that $v
\mapsto m(v) / \varphi(v)$ is strictly decreasing, $m(1)=1$, and with
\begin{equation*}
  \text{$m(v) > R$ \quad whenever \quad $\varphi(v) > 2R$.}
\end{equation*}
It is possible to find such $m$, since one may take
\begin{equation*}
  m(v) :=
  \begin{cases}
    \varphi(v)^{1-\epsilon} \qquad &\text{if $\varphi(v) < 2R$,}
    \\
    (2R)^{1-\epsilon} \qquad &\text{if $\varphi(v) \geq 2R$,}
  \end{cases}
\end{equation*}
for small enough $\epsilon > 0$. The quantities in Theorem
\ref{thm:subgeometric_Harris_discrete} can then be bounded as follows:
\begin{equation*}
  g(u) = \frac{1}{u} \psi'(u) \varphi(u)  = \frac{m(u)}{u}
  \geq \frac{1}{u}.
\end{equation*}
\begin{equation*}
  F_\psi(\zeta) = \int_\zeta^1 \frac{1}{h(u)} \d u
  \leq
  \int_\zeta^1 f^{-1}(\xi) \d \xi.
\end{equation*}
For example, if $\varphi(v) = v^{1-\alpha}$, $\alpha \in (0,1)$, we obtain (with $C$
standing for a positive constant)
\begin{gather*}
  \psi(u) \leq 1 + C u^\alpha,
  \qquad
  f(u) = \frac{\psi(u)}{u} \leq C u^{\alpha-1},
  \qquad
  f^{-1}(\xi) \leq C \xi^{\frac{1}{\alpha-1}},
  \\
  F_\psi(\zeta) \leq C (\zeta^{\frac{\alpha}{\alpha-1}} - 1),
  \qquad
  F_\psi^{-1}(t) \leq (1 + C t)^{1 - \frac{1}{\alpha}}.
\end{gather*}

As a conclusion, in that case, the rate of convergence \eqref{eq:HarrisSubgeo2} is 
  \begin{equation}
    \label{eq:HarrisSubgeo2-poly}
    \| S^n \nu \| \leq
    \frac{C}{n^{1/\alpha}}  \| \nu \|_{V}
    \qquad \text{for all $n \ge 1$},
  \end{equation}
  for any $\nu \in \cN_{V}$ and some explicitly computable constant $C  \geq 1$.

\paragraph{Exponential decay.} On the other hand, when  $\varphi(u) := {u / (\log u)^{\alpha}}$, $\alpha > 0$, we take $\psi(u) = u^\kappa$, $0 < \kappa < 1$. 
We next compute 
$$
f(v) = v^{\kappa-1}, \quad g(v) = \kappa v^{\kappa-1} (\log v)^{-\alpha}, \quad h (u) = C_1 u (\log u^{-1})^{-\alpha},
$$
for a constant $C_1 = C_1(\alpha,\kappa) \in (0,\infty)$, and finally 
$$
F(v) = C_2 (\log v^{-1})^{\alpha + 1}, \quad F^{-1}(u) = e^{-C_3  u^{1 \over \alpha +1}}, 
$$
for some constants  $C_i = C_i(\alpha,\kappa) \in (0,\infty)$. As a conclusion, in that case, the rate of convergence \eqref{eq:HarrisSubgeo2} is 
  \begin{equation}
    \label{eq:HarrisSubgeo2-expo}
    \| S^n \nu \| \leq C e^{-\lambda n^{1 \over \alpha +1}}  \| \nu \|_{V}
    \qquad \text{for all $n \ge 1$},
  \end{equation}
  for any $\nu \in \cN_{V}$ and some explicitly computable constants $C  \geq 1$, $\lambda \in (0,\infty)$.


\section{Results for continuous-time semigroups}
\label{sec:time-continuous}

In this section we again address the speed of relaxation to
equilibrium, this time in the framework of continuous-time
semigroups. The most straightforward results are obtained by applying
the discrete-time results in the previous sections to any stochastic
semigroup $(S_t)_{t \geq 0}$ as long as $S_T$ satisfies the needed
assumptions for some $T > 0$. We will state these results first. Then,
in the setting of continuous-time semigroups it is perhaps more
natural to look for similar Foster-Lyapunov-type conditions on the
generator of the semigroup instead of conditions on $S_T$ for a given
$T > 0$. Our main aim in this section is to prove results of this
type. We notice that they can be obtained as consequences of our
discrete-time results both in the geometric and subgeometric cases.

\subsection{Geometric convergence}
\label{sec:time-continuous-geometric}

First we state Doeblin's Theorem~\ref{thm:Doeblin}, as applied to a
continuous semigroup, with a straightforward proof:

\begin{thm}[Semigroup version of Doeblin's theorem]
  \label{thm:Doeblin-semigroup}
  Let $(S_t)_{t \geq 0}$ be a stochastic semigroup in $\M$. If there
  exists $T > 0$ such that $S_T$ satisfies the Doeblin
  condition \eqref{eq:Doeblin} then the semigroup $(S_t)_{t \geq 0}$ has a unique
  equilibrium $\mu^*$ in $\P$, and
  \begin{equation}
    \label{eq:Doeblin-semigroup-decay}
    \| S_t   \nu  \| \leq
    \frac{1}{1-\alpha} e^{- \lambda t} \| \nu \|,
    \qquad \text{for all $t \geq 0$},
  \end{equation}
  for all $\nu \in \cN$, where
  \begin{equation*}
    \lambda := -\frac{\log (1-\alpha)}{T} > 0.
  \end{equation*}
\end{thm}

\begin{proof}[Proof of Theorem~\ref{thm:Doeblin-semigroup}]
  Theorem \ref{thm:Doeblin} shows that the operator $S_{T}$ has a
  unique stationary state in $\P$, which we call $\mu^*$. In fact, $\mu^*$
  is a stationary state of the whole semigroup since, for all $s \geq
  0$, we have
  \begin{equation*}
     S_{T} S_{s} \mu^* = S_{s} S_{T} \mu^* = S_{s} \mu^*,
   \end{equation*}
   which shows that $S_{s} \mu^*$ (which is again a probability measure)
   is also a stationary state of $S_{T}$. Due to uniqueness, we deduce
   \begin{equation*}
     S_s \mu^* = \mu^*
     \qquad \text{for all $s \geq 0$}.
   \end{equation*}
   This stationary state is clearly unique in $\P$, since any
   stationary state of $(S_t)_{t \geq 0}$ is in particular a
   stationary state of $S_{T}$.

   In order to show \eqref{eq:Doeblin-semigroup-decay}, for any
   $\nu \in \cN$ and any $t \geq 0$ we write
$
     k := \lfloor{t/T}\rfloor,
$
   where $\lfloor \cdot \rfloor$ is the floor function, so that
   \begin{equation*}
     \frac{t}{T} - 1 < k \leq \frac{t}{T}.
   \end{equation*}
   Then,
  \bean
   \| S_t \nu \|
&=& \| S_{t - k T} S_{k  T} \nu \| \leq
     \| S_{k T} \nu \|
  \\
     &\leq&
     (1-\alpha)^{k} \| \nu \|
     \leq
     \frac{1}{1-\alpha}
     \exp\left( 
       \frac{t \log (1-\alpha)}{T} \right) \| \nu \|,
\eean  
   which is nothing but \eqref{eq:Doeblin-semigroup-decay}.     
\end{proof}

As above, we could write the immediate counterpart of Harris's Theorem
\ref{thm:Harris}, as applied to a continuous semigroup.  We rather
present a version more adapted to a semigroup setting. Indeed, in the
continuous time setting, it is natural to consider Foster-Lyapunov
conditions on the generator $\Lambda$ of the semigroup
$(S_t)_{t \geq 0}$.  A natural assumption that replaces the operator
Lyapunov condition in Hypothesis \ref{hyp:op_Lyapunov} is
\begin{equation}
  \label{eq:Lyapunov-generator}
  LV \leq -\sigma V + b,
\end{equation}
for some constants $\sigma, b > 0$ and some continuous weight
(Lyapunov) function $V \: \Omega \to [1,+\infty)$, where $L$ is dual
to the generator $\Lambda$. When $(S_t)_{t \geq 0}$ is of Feller type,
we have $\Lambda = L^*$ and $L$ is the generator of the associated
Markov-Feller semigroup $(P_t)_{t \geq 0}$ on $C_0(\Omega)$.  This
faces the technical problem that the generator $L$ may not be defined
on the particular functions $V$ we wish to consider.  However, for
$0 \le \mu \in \M_V$, we may compute at least formally
\beqn\label{eq:dtStmu=} {d \over dt} \int V S_t \mu = \int LV S_t \mu
\le \int (-\sigma V + b) S_t \mu.  \eeqn After time integration, we
thus get (still formally)
\begin{equation}
  \label{eq:Lyapunov-generator-mild-S}
\int V S_t \mu  \le \int V \mu + \int_0^t \!\! \int (-\sigma V + b)  S_s \mu \d s, 
\end{equation}
for all $t \geq 0$ and all $0 \le \mu \in \M_V$. Equation
\eqref{eq:Lyapunov-generator-mild-S} is a common way to understand the
generator Lyapunov condition \eqref{eq:Lyapunov-generator} and thus to
avoid the difficulty of defining $LV$. Let us also notice that this
problem has been circumvented in different ways in other works: for
diffusion semigroups, the generator is a local operator which is
naturally defined on arbitrary $\mathcal{C}^2$ functions $V$
\citep{BGL14}; for general semigroups, one may define
\eqref{eq:Lyapunov-generator} to mean that the process
$V(X_t) - \int_0^t (-\sigma V(X_s) + b) \d s$ is a supermartingale for
every starting condition $x_0$ (where $(X_t)_{t \geq 0}$ is the
process associated to the semigroup $(S_t)_{t \geq 0}$), a path which
is taken for example in \citet{GDF2009, Hairer2016notes}. This
probabilistic formulation is equivalent to saying that the associated
Markov-Feller semigroup $(P_t)_{t \geq 0}$ satisfies \begin{equation}
  \label{eq:Lyapunov-generator-mild}
  P_t V \leq V + \int_0^t P_s (-\sigma V + b) \d s,
  \qquad \text{for $t \geq 0$,}
\end{equation}
or equivalently, that the associated stochastic semigroup
$(S_t)_{t \geq 0}$ satisfies \eqref{eq:Lyapunov-generator-mild-S}.
Another possible alternative formulation, which we will not use in
this work, is to require that \eqref{eq:dtStmu=} effectively holds for
any positive $\mu$ which belongs to the domain $D_V(L)$ defined by
$$
D_V(L) := \Bigl\{ \mu \in \M_V; \, \lim_{t \searrow 0} \int{S_t \mu - \mu \over t} \phi \ \hbox{ exists, for any } \ \phi \in C(\Omega), 
\ \phi/V \ \hbox{bounded} \Bigr\},
$$
provided that this one is dense in $\M_V$.

One can take this a step further by observing that, again at least
formally, one may use Gronwall's lemma on
\eqref{eq:Lyapunov-generator-mild-S} in order to deduce the bound
\begin{equation}
  \label{eq:Lyapunov-semigroup-S}
    \| S_t \mu \|_V \leq e^{-\sigma t} \| \mu \|_V + \frac{b}{\sigma}
    (1- e^{-\sigma t}) \| \mu \|
    \qquad \text{for all $t \geq 0$ and all $\mu \in \M_v$.}
  \end{equation}
  In a first result, we
  choose to avoid these technical problems altogether and state as an
  assumption the specific consequence we need from either
  \eqref{eq:Lyapunov-generator} or \eqref{eq:Lyapunov-generator-mild},
  which is the following:

\begin{hyp}[Semigroup Lyapunov]
  \label{hyp:semigroup_Lyapunov}
  Let $V \: \Omega \to [1,+\infty)$ be a measurable function and
  $(S_t)_{t \geq 0}$ a stochastic semigroup on $\M_V$. We say the
  semigroup $(S_t)_{t \geq 0}$ satisfies the \emph{semigroup Lyapunov
    condition} with function $V$ when there exist constants $\sigma, b
  > 0$ such that \eqref{eq:Lyapunov-semigroup-S} holds. 
\end{hyp}
We will later give a specific hypothesis on the dual generator $L$
which ensures this condition holds (see Hypothesis
\ref{hyp:generator_Lyapunov}). However, we believe it is useful to
state the basic condition in Hypothesis \ref{hyp:semigroup_Lyapunov},
since in concrete applications it may well happen that
\eqref{eq:Lyapunov-semigroup-S} can be proved in some other way.
 
 \begin{thm}[Semigroup version of Harris's theorem]
  \label{thm:Harris-semigroup}    Let $V \: \Omega \to [1,+\infty)$ be a measurable (weight) function
  and let $(S_t)_{t \geq 0}$ be a stochastic semigroup in
  $\M_V$. Assume that
  \begin{enumerate}
  \item The semigroup $(S_t)_{t \geq 0}$ satisfies the semigroup
    Lyapunov condition (Hypothesis \ref{hyp:semigroup_Lyapunov}).
    
  \item For some $T > 0$, $S_T$ satisfies the local coupling condition
    (Hypothesis \ref{hyp:local_coupling}) with ${b}/A < \sigma$.
  \end{enumerate}
  Then the semigroup has an invariant probability measure
  $\mu^* \in \P_V$ which is unique within $\P_V$, and there exist
  $\lambda, C > 0$ such that
  \begin{equation}
    \label{eq:Harris-semigroup-decay}
    \| S_t \nu \|_{V} \leq
    C e^{-\lambda t} \| \nu \|_{V},
    \qquad \text{for $t \geq 0$},
  \end{equation}
  for all $\nu \in \cN_V$.
\end{thm}

\begin{proof}[Proof of Theorem~\ref{thm:Harris-semigroup}]
  Hypothesis \ref{hyp:semigroup_Lyapunov} shows that the operator
  Lyapunov condition (Hypothesis \ref{hyp:op_Lyapunov}) holds for
  $S_T$, since
    \begin{equation*}
    \| S_T \mu \|_V \leq e^{-\sigma T} \| \mu \|_V + \frac{b}{\sigma}
    (1- e^{-\sigma T}) \| \mu \|, 
  \end{equation*}
  for all $\mu \in \M_V$. The condition ${b}/A < \sigma$ hence ensures
  $S_T$ is in the conditions of Theorem \ref{thm:Harris}.
  
    With the same reasoning as in the proof of Theorem~\ref{thm:Doeblin-semigroup},
   we see that $(S_t)_{t \geq 0}$ has a
  unique stationary state in $\P_V$, which we call $\mu^*$. 
  We know from Theorem \ref{thm:Harris} that there exist a new norm $\Nt \cdot \Nt_{V}$ 
  (defined through \eqref{eq:newNormV} and a parameter $\beta > 0$, which is equivalent to the norm $\| \cdot \|_{V}$)
   and $0 < \gamma <  1$ such that
  \begin{equation*}
 \Nt S_{T} \nu \Nt_{V}    \leq \gamma \Nt   \nu \Nt_{V}, 
  \end{equation*}
  for all measures $\nu \in \cN_V$.
  In order to show
  \eqref{eq:Harris-semigroup-decay}, we follow a similar reasoning as
  in the proof of Theorem \ref{thm:Doeblin-semigroup}. Notice first
  that due to \eqref{eq:Doeblin-contrac1} and \eqref{eq:V-semigroup-growth} we have, for
  $0 \leq t \leq T$,
  \begin{equation}
    \label{eq:beta-semigroup-growth}
   \Nt S_t \mu \Nt_V
    \leq C_V e^{\omega_V T}  \Nt \mu \Nt_V, 
      \end{equation}
  for all measures $\mu \in \M_V$.  We conclude that \eqref{eq:Harris-semigroup-decay} holds with
    \begin{equation*}
    C := \frac{C_V e^{\omega_VT} }{\gamma} \frac{1+\beta}{\beta},
    \qquad
    \lambda := -\frac{\log \gamma}{T} > 0,
  \end{equation*}
  and $\gamma, \beta$ are the constants in Theorem \ref{thm:Harris} as
  applied to the operator $S_{T}$.
   \end{proof}
 
   We end this section by noticing that in the case
   of a Feller-type semigroup, Hypothesis \ref{hyp:semigroup_Lyapunov}
   is a consequence of the following condition on the dual generator
   $L$ of the associated Markov-Feller semigroup $(P_t)_{t \geq 0}$ on
   $C_0(\Omega)$. Hence Theorem
   \ref{thm:Harris-semigroup} also holds it $(S_t)_{t \geq 0}$ is a
   Feller-type stochastic semigroup which satisfies the following
   hypothesis instead of Hypothesis \ref{hyp:semigroup_Lyapunov}:
 
\begin{hyp}[Generator Lyapunov]
  \label{hyp:generator_Lyapunov}
  Let
  $(S_t)_{t \geq 0}$ be a Feller-type stochastic semigroup on $\M_V$. We say the
  semigroup $(S_t)_{t \geq 0}$ satisfies the \emph{generator Lyapunov
    condition} with function $V$ when there exist constants $\sigma, b
  > 0$ such that \eqref{eq:Lyapunov-generator-mild-S} holds. 
\end{hyp}

The fact that in the case of a Feller-type stochastic semigroup,
Hypothesis \ref{hyp:generator_Lyapunov} implies Hypothesis
\ref{hyp:semigroup_Lyapunov} is a straightforward consequence of the
following version of Gronwall's lemma:

\begin{lem}[Gronwall lemma]
  \label{lem:Feller-Harris-semigroup}
  Consider $(S_t)_{t \geq 0}$ a Feller-type stochastic semigroup in
  $\M_V$ which satisfies the generator Lyapunov condition (Hypothesis
  \ref{hyp:generator_Lyapunov}) associated to $V$ and some constants
  $\sigma,b > 0$.  Then $(S_t)_{t \geq 0}$ satisfies the corresponding
  \emph{semigroup Lyapunov condition} (Hypothesis
  \ref{hyp:semigroup_Lyapunov}) with the same function $V$ and
  constants $\sigma, b > 0$.
\end{lem}

%
%
%

We observe that the difficulty in proving this result is that there is
no reason why the function $t \mapsto \int \mu_t V$ should be
continuous, which makes it difficult to apply standard results on
differential inequalities, which usually require a continuous
function.

\begin{proof}[Proof of Lemma~\ref{lem:Feller-Harris-semigroup}]
We fix $0 \le \mu_0 \in \M_V$ and we set $\mu_t := S_t \mu_0$. 
We split the proof into four steps. 

\smallskip\noindent
{\sl Step 1. } We first observe that the Lyapunov condition \eqref{eq:Lyapunov-generator-mild-S} is equivalent to 
\begin{equation}
  \label{eq:LyapInEqInt}
\int \mu_{t_2} V + \sigma \int_{t_1}^{t_2} \!\! \int \mu_s V ds \le  \int \mu_{t_1} V + b \int_{t_1}^{t_2} \!\! \int \mu_s   ds, 
 \end{equation}
for any $t_2 > t_1 \geq 0$ and all $0 \le \mu \in \M_V$. The inequality \eqref{eq:Lyapunov-generator-mild-S} is indeed
a particular case of \eqref{eq:LyapInEqInt} and the reciprocal implication is an immediate consequence of
the semigroup property of $(S_t)_{t \ge 0}$.

\smallskip\noindent
{\sl Step 2. } We claim that
\beqn\label{eq:cad}
t \mapsto  \int V \mu_t \ \hbox{ is c\`ad}.
\eeqn
We recall that because $(S_t)_{t \geq 0}$ is of Feller-type, there holds 
\beqn\label{eq:C0St}
t \mapsto \int \mu_t \chi  \in C(\R_+;\R_+), 
\eeqn
for any $\chi \in C_c(\Omega)$, and even for any $\chi \in C_b(\Omega)$. 
On the one hand, as a consequence of \eqref{eq:C0St}, for any $\chi \in C_c(\Omega)$, $\chi \le V$, there holds 
$$
\int \mu_t \chi = \lim_{s \to t} \int \mu_s \chi \le \liminf_{s \to t} \int \mu_s V.
$$
Choosing $\chi_n \nearrow V$, the monotone convergence theorem implies
\beqn\label{eq:cad1} \int \mu_tV = \lim_{n \to \infty} \int \mu_t
\chi_n \le \liminf_{s \to t} \int \mu_s V.  \eeqn On the other hand,
the semigroup Lyapunov condition \eqref{eq:LyapInEqInt} implies
$$ \int V \mu_{s} \le \int V \mu_{t} + b (s-t) \int \mu_0, \quad
\forall \, s > t.
$$
We deduce that 
\beqn\label{eq:cad2}
\limsup_{s \searrow t} \int V \mu_{s} \le \lim_{s \searrow t} \Bigl\{ \int V \mu_{t} + b (s-t) \int \mu_0 \Bigr\} = \  \int V \mu_{t}. 
\eeqn
Equations \eqref{eq:cad1} and \eqref{eq:cad2} together imply \eqref{eq:cad}. 
 
\smallskip\noindent
{\sl Step 3. } We claim that the Lyapunov
condition \eqref{eq:Lyapunov-generator-mild-S} (or equivalently
\eqref{eq:LyapInEqInt}) is equivalent to the fact that \eqref{eq:cad}
holds together with
\begin{equation}
\label{eq:LyapEDO}
\ddt \int V \mu_t  \le
- \sigma  \int  V \mu_t + b \int   \mu_t, 
\end{equation}
in the sense of $\D'(0,+\infty)$, the space of distributions on
$(0,+\infty)$. On the one hand, if we assume \eqref{eq:LyapInEqInt}
holds then \eqref{eq:cad} holds from Step 2. Multiplying equation
\eqref{eq:LyapInEqInt} by a nonnegative test function
$\varphi \in \D(0,+\infty)$, dividing it by $t_2-t_1$, integrating and
passing to the limit as $t_2 \searrow t_1$ (and using \eqref{eq:cad}
to do this), we deduce \eqref{eq:LyapEDO}.

On the other hand, if we assume that both \eqref{eq:cad} and
\eqref{eq:LyapEDO} hold, in particular this means that
\beqn\label{eq:LyapEDOBis} \int_0^\infty \int V \mu_s \phi'_s \d s +
\sigma \int_0^\infty \phi_s \int V \mu_s \d s \le b \int_0^\infty
\phi_s \int \mu_s \d s, \eeqn for any $0 \le \phi \in
\D(0,+\infty)$.
For $t > 0$ and a given function $0 \le \rho \in \D(\R_+)$ with
integral $1$ and $\supp \rho \subset (0,1)$, define the sequence
$(\phi_n)$ by $\phi_n(0) := 0$ and
$\phi'_n (s) := n \rho(s/n) - n \rho((s-t)/n)$. We may pass to the
limit $n \to \infty$ in \eqref{eq:LyapEDOBis} by taking advantage of
\eqref{eq:cad}, and we conclude \eqref{eq:Lyapunov-generator-mild-S}.

 \smallskip\noindent
{\sl Step 4. } 
We introduce a mollifier $(\rho_\eps)$ with $\supp (\rho_\eps) \subset (-\eps,0)$ and the function 
$$
u_\eps (t)  = (\| \mu \|_V * \rho_\eps)(t) = \int_\R \| \mu_s \|_V \rho_\eps(t-s) \, ds,
$$
which clearly satisfies $u_\eps \in C^1$. From \eqref{eq:LyapEDO}, $u_\eps$ also clearly satisfies  
$$
u'_\eps \le - \sigma u_\eps + K \int   \mu_0, 
$$
pointwise on $(0,\infty)$. 
From the classical version of the Gronwall lemma, we  deduce that 
$$
u_\eps(t_2) \le e^{-\sigma (t_2-t_1)} u_\eps(t_1)  + {K \over \sigma} (1-e^{-\sigma (t_2-t_1)}) \int \mu_{0} , 
$$
for any $t_2 > t_1 > 0$ and any $\eps > 0$. 
Observing that $u_\eps(t) \to \| \mu_t\|_V$ as $\eps \to 0$ for any $t \ge 0$ because of \eqref{eq:cad}, we obtain that the semigroup Lyapunov condition  \eqref{eq:Lyapunov-semigroup-S} holds for any $t = t_2 > 0$,
by passing to the limit  $\eps \to 0$ and next $t_1 \to 0$. 
\end{proof}

\subsection{Subgeometric convergence}
\label{sec:Harris-generator}

\quad\ As we have just done for the geometric case, one may state
analogous results to Theorems~\ref{theo:HarrisSubgeo1} or
\ref{thm:subgeometric_Harris_discrete} in the case of a continuous
semigroup $(S_t)_{t \geq 0}$, as long as the conditions of the
theorems are satisfied by $S_T$ for some time $T > 0$. However, the
conditions in Theorems~\ref{theo:HarrisSubgeo1} and
\ref{thm:subgeometric_Harris_discrete} are not so natural for a
continuous semigroup, since they involve estimates for $S_T$ and
powers of $S_T$.

We will avoid these statements and give more
convenient conditions in terms of the semigroup, in
the spirit of Hypothesis \ref{hyp:semigroup_Lyapunov}, 
and in terms of the generator of the semigroup, in
the spirit of Hypothesis \ref{hyp:generator_Lyapunov}.

A natural weak counterpart of the  Lyapunov condition \eqref{eq:Lyapunov-generator} consists in assuming that 
\begin{equation}
  \label{eq:weak-Lyapunov-generator}
  L V \leq -\sigma \varphi(V) + b,
\end{equation}
for some measurable weight (Lyapunov) function
$V \: \Omega \to [1,+\infty)$, some concave function
$\varphi\: [1,+\infty) \to [1,+\infty)$ and some constants
$\sigma, b > 0$, where $L$ is adjoint to the generator $\Lambda$ of
$(S_t)_{t \ge 0}$. This runs into the same technical problems
discussed before Hypothesis \ref{hyp:semigroup_Lyapunov}, so we will
again use the consequence we would like to extract as an assumption,
and leave it to be checked in each specific application.  Proceeding
similarly from \eqref{eq:weak-Lyapunov-generator} as for
\eqref{eq:Lyapunov-generator}, we may formally compute
\begin{equation}
  \label{eq:weak-gen-Lyapunov-hyp}
  \|S_t \mu \|_V \leq
  \| \mu \|_V
  + \int_0^t ( b \|S_u \mu\| -   \sigma \| S_u \mu \|_{\varphi(V)}) \, \d u, 
\end{equation}
for any $0 \le \mu \in \M_V$ and any $t \ge 0$. For a Feller-type
semigroup, this condition is equivalent to conditions (3.1) or (3.2)
in \citet{GDF2009}.  We can take one more step and deduce from
\eqref{eq:weak-gen-Lyapunov-hyp} (still formally) the weak confinement counterpart of
\eqref{eq:Lyapunov-semigroup-S}.   As we will establish for a
Feller-type stochastic semigroup (see
Corollary~\ref{cor:weak-gen-Gronwall} below), a natural consequence is
\begin{equation}
  \label{eq:weak-SG-Lyapunov-hyp}
  \| S_t \mu \|_{V} + \sigma t \|S_t \mu\|_{\varphi(V)}
  \leq \|\mu\|_{V} +  K_t   \|\mu\|,
\end{equation}
for all $t \ge 0$ and $\mu \in \M_{V}$, with
$K_t := t b (1 + \sigma t/2)$.  We then take this last property as the
assumption we impose on the semigroup:
\begin{hyp}[Weak semigroup Lyapunov condition]
  \label{hyp:weak_semigroup_Lyapunov}
  We say that a stochastic semigroup $(S_t)_{t \ge 0}$ satisfies the
  weak semigroup Lyapunov condition for a weight function
  $V \: \Omega \to [1,+\infty)$ and a concave function
  $\varphi \: [1,+\infty) \to [1,+\infty)$, $\varphi(v) \le v$ for any
  $v \ge 1$, if there exist constants $b, \sigma > 0$ such that
  \eqref{eq:weak-SG-Lyapunov-hyp} holds.
\end{hyp}

The following continuous-time analogue of Theorem
\ref{theo:HarrisSubgeo1} is our main result in this setting:
 
\begin{thm}[subgeometric Harris, interpolated version]
  \label{th:sHarris-interpolated}
  Consider two measurable weight functions $V_1, V_2 \: \Omega \to [1,+\infty)$, $V_1 \le V_2$, 
  and  a stochastic semigroup $(S_t)_{t \geq 0}$ on
  $\M_{V_2}$  such that:
  \begin{enumerate}
  \item the weak semigroup Lyapunov condition
    (Hypothesis \ref{hyp:weak_semigroup_Lyapunov}) holds for both weights $V_1$
    and $V_2$, with functions and constants $\varphi_1$, $b_1$, $\sigma_1$ and
    $\varphi_2$, $b_2$, $\sigma_2$.
    
  \item For some time $T>0$, $S_T$ satisfies the local coupling
    condition (Hypothesis \ref{hyp:local_coupling}) for both
    $\varphi_1(V_1)$ and $\varphi_2(V_2)$, with constants
   $A_1 > K_1/\sigma_1$ and $A_2 > K_2 / \sigma_2$. 
    
  \item  The interpolation condition \eqref{eq:InterpolationV1V0V3} holds for some $\xi$. \Black  
  \end{enumerate}
  Then there exists a unique equilibrium $\mu^* \in \P_{\varphi_2(V_2)}$, 
  and there exists some
  $C > 0$ depending only on the constants in the assumptions such that
  \begin{equation*}
    \| S_t \nu \|_{V_1} \lesssim  C \Theta(t) \| \nu \|_{V_2}, \quad \forall \, t \ge 0, 
  \end{equation*}
  and 
  \begin{equation*}
    \| S_t \nu \|  \lesssim  C \widetilde \Theta (t) \| \nu \|_{V_2}, \quad \forall \, t \ge 0, 
  \end{equation*}
  for any $\nu \in \cN_V$, where
    \begin{equation*}
      \Theta(t) := F^{-1} (t),
      \qquad
      \widetilde{\Theta}(t) := \frac{1}{t} F^{-1}(\frac{t}{2}),
      \qquad
      F(v) := \int_v^1 \frac{1}{\xi^*(u)} \d u.
  \end{equation*}
\end{thm}  

\smallskip 
Before coming to the proof of Theorems \ref{th:sHarris-interpolated}, we present a technical 
result.

\begin{lem}
  \label{lem:implicit-explicit}
  Let $V \: \Omega \to [1,+\infty)$ be a weight function and
  $\varphi \: [1,+\infty) \to [1,+\infty)$ be a continuous and increasing
  function with $\varphi(1)=1$. Let $S$ be a stochastic operator which
  satisfies the following ``implicit'' Lyapunov-type condition:
  \begin{equation}
    \label{eq:implicit-Lyapunov}
    \| S \mu \|_{V} + \sigma \|S \mu\|_{\varphi(V)}
    \leq \|\mu\|_{V} + K \|\mu\|,
    \qquad \text{for all $\mu \in \P \cap \M_{V}$}
  \end{equation}
  for some $K, \sigma > 0$. If we define
  $\widetilde{V} \: \Omega \to [1,+\infty)$ and
  $\widetilde{\varphi} \: [1, +\infty) \to [1,+\infty)$ by
  \begin{equation*}
    \widetilde{V} := \frac{1}{1 + \sigma}(V + \sigma \varphi(V)),
    \qquad
    \widetilde{\varphi}(\widetilde{V})
    := \varphi(V)
  \end{equation*}
  then $\widetilde{\varphi}$ is increasing,
  $\widetilde{\varphi}(1)=1$, and $S$ satisfies  a usual weak
  explicit Lyapunov condition for $\widetilde{V}$ and
  $\widetilde{\varphi}$, namely 
  \begin{equation}
    \label{eq:explicit-Lyapunov}
    \| S \mu \|_{\widetilde{V}}
    + \frac{\sigma}{1+\sigma} \|\mu\|_{\widetilde{\varphi}(\widetilde{V})}
    \leq \|\mu\|_{\widetilde{V}} + \frac{K}{1+\sigma} \|\mu\|,
    \qquad \text{for all $\mu \in  \M_{V}$}.
  \end{equation}
\end{lem}

\begin{proof}[Proof of Lemma~\ref{lem:implicit-explicit}]
  First, note that $\widetilde{\varphi}$ is well defined and
  $\widetilde{\varphi}(1) = 1$, since
  $V \mapsto \frac{1}{1 + \sigma}(V + \sigma \varphi(V))$ is a
  strictly increasing function which takes the value $1$ for
  $V=1$. Using \eqref{eq:implicit-Lyapunov} and the definition of
  $\widetilde{V}$ we have, for any $\mu \in  \M_{V}$,
  \begin{multline*}
    (1+\sigma) \| S \mu \|_{\widetilde{V}}
    = \|S \mu\|_V + \sigma \|S \mu\|_{\varphi(V)}
    \\
    \leq
    \| \mu \|_V + K \|\mu \|
    = (1+\sigma) \|\mu\|_{\widetilde{V}}
    - \sigma \|\mu\|_{\varphi(V)} + K \|\mu\|.
  \end{multline*}
  Due to the definition of $\widetilde{\varphi}$, this is precisely
  \eqref{eq:explicit-Lyapunov}.
\end{proof}

\bigskip

\begin{proof}[Proof of Theorem \ref{th:sHarris-interpolated}]
  Let us show that the conditions of Theorem \ref{theo:HarrisSubgeo1}
  are met by $S_{t_0}$ for a certain $t_0 > 0$. First, for any
  $t_0 > 0$, we have the following implicit Lyapunov-type inequalities
  by assumption:
  \begin{equation*}
  \| S_{t_0} \mu \|_{V_i} + \sigma_i t_0 \|S_{t_0} \mu\|_{\varphi_i(V_i)}
    \leq \|\mu\|_{V_i} + K_it_0 (1 +  \sigma_i t_0/2) \|\mu\|,
  \end{equation*}
  for $i = 1, 2$ and all $\mu \in \P \cap \M_{V_i}$. We may define
  \begin{equation*}
    \widetilde{V}_i := \frac{1}{1 + \sigma_i}(V_i + \sigma_i \varphi_i(V_i)),
    \qquad
    \widetilde{\varphi}_i(\widetilde{V}_i)
    := \varphi(V_i),
  \end{equation*}
  and we know from Lemma \ref{lem:implicit-explicit} that we also have
  the weak Lyapunov condition:
  \begin{equation}
    \label{eq:wL}
    \| S_{t_0} \mu \|_{\widetilde{V}_i} + \frac{\sigma_i t_0}{1+
      \sigma_i t_0} \|\mu\|_{\widetilde{\varphi}_i(\widetilde{V}_i)}
    \leq \|\mu\|_{\widetilde{V}_i}
    + \frac{K_i t_0}{1 + \sigma_i t_0} (1 + t_0/2) \|\mu\|.
  \end{equation}
  Choose an integer $N > 0$ and take $t_0 := T/N$. We can choose $N$
  large enough so that
  \begin{equation}
    \label{eq:10}
   \frac{K_i}{\sigma_i} (1 +  \sigma_i t_0 / 2) < A_i, 
    \quad \text{for $i=1,2$,}
  \end{equation}
  and then all hypotheses of Theorem \ref{theo:HarrisSubgeo1} are
  satisfied by the operator $S_{t_0}$, since
  \begin{enumerate}
  \item $S_{t_0}$ satisfies the weak Lyapunov condition \eqref{eq:wL}
    for $\widetilde{V}_1$, $\widetilde{\varphi}_1$ and
    $\widetilde{V}_2$, $\widetilde{\varphi}_2$.
  \item $S_{t_0}^N = S_T$ satisfies the local coupling condition for
    both $\widetilde{\varphi}_1$ and $\widetilde{\varphi}_2$, with
    constants which satisfy the appropriate inequality thanks to
    \eqref{eq:10}.
    
  \item The interpolation condition \eqref{eq:InterpolationV1V0V3}
    and the assumption that $\varphi_1(V_1) \leq V_1$ show that
    \begin{multline*}
      \lambda \widetilde{V}_1
      = \frac{\lambda}{1 + \sigma_1} (V_1 + \sigma_1 \varphi_1(V_1))
      \\
      \leq \lambda V_1 \leq \varphi_1(V_1) + \xi(\lambda) V_2
      \leq \widetilde{\varphi}_1(\widetilde{V}_1)
      + (1+\sigma_2)\xi(\lambda) \tilde{V}_2.
    \end{multline*}
    Hence the interpolation condition is satisfied for
    $\widetilde{\xi}(\lambda) := (1+\sigma_2) \xi(\lambda)$.
  \end{enumerate}
  Applying Theorem \ref{theo:HarrisSubgeo1} gives an estimate of the
  decay of $\| S_{t} \mu \|_V$ and $\| S_{t} \mu\|$ for $t = n
  t_0$. The same technique used before in the proof of Theorem
  \ref{thm:Doeblin-semigroup} allows us to extend the decay to
    the whole semigroup and obtain the result.
\end{proof}

We end this section by specifying Harris' theorem  to the case of a Feller-type semigroup for which some
simplifications occur. In this setting, the relevant confinement condition writes: 

\begin{hyp}[Weak generator Lyapunov condition]
  \label{hyp:weak_generator_Lyapunov}
  We say that a  Feller-type stochastic   semigroup $(S_t)_{t \ge 0}$ satisfies the weak generator Lyapunov condition for a
  weight continuous function $V \: \Omega \to [1,+\infty)$ if there exist
  constants $b, \sigma > 0$ and a continuous function
  $\varphi \: [1,+\infty) \to [1,+\infty)$ such that
  \eqref{eq:weak-gen-Lyapunov-hyp} holds.
\end{hyp}

In much the same way as in Section \ref{sec:subgeometric-discrete},
using Theorem \ref{th:sHarris-interpolated} we can prove the following
result, which is a close relative of the main result in
\citet{GDF2009}:

\begin{thm}[Subgeometric Harris]
  \label{thm:sHarris}
  Consider a Feller-type stochastic semigroup $(S_t)_{t \geq 0}$ on $\M_{V}$ which
  satisfies both the the weak generator Lyapunov condition for a
  continuous weight function $V$ (Hypothesis \ref{hyp:weak_generator_Lyapunov}) and the Harris
  irreducibility condition (Hypothesis \ref{hyp:Harris}) on the set
  $\mathcal{C} := \{ x\in \Omega \mid V(x) \leq R\}$, for large enough
  $R$. Then, there exists a unique equilibrium
  $\mu^* \in \P_{\varphi(V)}$, and there exist some constructive
  constant $C > 0$ and  a decay rate function
  $ \widetilde \Theta$ such that
  $$
  \| S_t \nu \|  \leq   \widetilde \Theta (t)  \| \nu \|_{V}, \quad \forall \, t \ge 0, 
  $$
  for any $\nu \in \cN_V$, where $\widetilde \Theta (t) := C \Theta_\psi(r t)/t$ with the notations of Theorem~\ref{thm:subgeometric_Harris_discrete}.
\end{thm}

The proof of Theorem~\ref{thm:sHarris} is given in the rest of this section.

\begin{lem}
  \label{lem:weak-gen-Gronwall}
  Let $V \: \Omega \to [1,+\infty)$ be a continuous weight function and
  $\varphi \: [1,+\infty) \to [1,+\infty)$ a concave
  function with $\varphi(1)=1$.  If  a Feller-type stochastic semigroup $(S_t)$ satisfies the  weak generator Lyapunov condition \eqref{eq:weak-Lyapunov-generator} then it satisfies 
\beqn\label{eq:weak-Lyap-gen-psi}
L  \psi(V)  \le - \psi'(V) \varphi(V) + \psi'(V) b, 
\eeqn
for any concave function $\psi \: [1,+ \infty) \to [1,+\infty)$. Both conditions have to be understood when integrated along the semigroup flow and thus $L$ denotes the generator of the
associated Feller-Markov semigroup $(P_t)_{t \ge 0}$ such that $S_t = P_t^*$.
  \end{lem}

\begin{proof}[Proof of Lemma~\ref{lem:weak-gen-Gronwall}]

For the same reason as in the geometric case, we have \eqref{eq:cad}. As a consequence, we have at least 
$$
t \mapsto  \int \psi(V) \mu_t \ \hbox{ is c\`ad}, 
$$
or even it is continuous when $\psi(s)/s \to 0$ as $s\to\infty$.
On the other hand, for any $0 \le \mu \in \M_V$,  we have 
$$
\int  \mu_0 \bigl\{ P_{t_2} V + \sigma \int_{t_1}^{t_2} P_s \varphi(V) ds \bigr\} \le  \int \mu_0  \bigl\{ P_{t_1} V + b ({t_2}- {t_1})  \bigr\}, 
$$
which is nothing but the dual form of \eqref{eq:weak-gen-Lyapunov-hyp}, 
 so that 
  \beqn\label{eq:LyapInEqIntweakP}
 P_{t_2} V + \sigma \int_{t_1}^{t_2} P_s \varphi(V) ds
 \le P_{t_1} V + b ({t_2}- {t_1}).
 \eeqn
Using Jensen's inequality \eqref{eq:jensen2} and
\eqref{eq:LyapInEqIntweakP}, for any $h > 0$ we have 
\begin{align*}
P_h \psi(V) 
&\le \psi(P_h V) \le 
\psi \bigl(V + b h - \sigma \int_0^h P_s \varphi(V) ds \bigr)
\\
&\le \psi (V)  + \psi'(V) \bigl(b h - \sigma \int_0^h P_s \varphi(V) ds \bigr).
\end{align*}
By duality,  for any $0 \le \mu \in \M_V$, we deduce
\bean 
\int (S_h \mu_t) \psi(V) -  \int  \mu_t \psi(V)    \le b\int_0^h \int  (S_s (\psi'(V) \mu_t)) ds - \sigma \int_0^h\int  (S_s (\psi'(V) \mu_t)) \varphi(V) ds, 
\eean
for any $h > 0$ and $t\ge0$. Dividing by $h>0$ and passing to the limit $h \to 0$, we get 
$$
{d \over dt} \int  \mu_t \psi(V) + \sigma \int   \mu_t  \psi'(V) \varphi(V) \le b  \int   \psi'(V) \mu_0, 
$$
which is the rigorous definition of the weak generator Lyapunov condition \eqref{eq:weak-Lyap-gen-psi}.
 \end{proof}
 
 \begin{cor}
  \label{cor:weak-gen-Gronwall}
 If $(S_t)$ is a Feller-type stochastic semigroup which satisfies the  weak generator Lyapunov condition \eqref{eq:weak-Lyapunov-generator} then it satisfies
   \begin{equation}
    \| S_{t} \mu \|_{V} + \sigma t \|S_{t} \mu\|_{\varphi(V)}
    \leq \|\mu\|_{V} + b t (1 + \sigma t/2) \|\mu\|.
  \end{equation}

  \end{cor}

\begin{proof}[Proof of Corollary~\ref{cor:weak-gen-Gronwall}]
Because of the  weak generator Lyapunov condition \eqref{eq:weak-Lyapunov-generator}  and the non-expansive mappings property   \eqref{eq:Doeblin-contrac1}, we have  
$$
   \| S_{t} \mu \|_{V} + \sigma \int_0^t \|S_{u} \mu\|_{\varphi(V)} \d u
    \leq \|\mu\|_{V} + b t   \|\mu\|,
$$
for any $t \ge 0$. 
On the other hand, because of Lemma~\ref{lem:weak-gen-Gronwall} applied to $\psi := \varphi$, we have 
$$
   \| S_{t} \mu \|_{\varphi(V)} + \sigma \int_u^t \|S_{u} \mu\|_{\varphi'(V)\varphi(V)} \d u 
    \leq \|S_u \mu\|_{\varphi(V)} + b (t-u)   \|\mu\|.
$$
After time integration of that last estimate and throwing away the second term at the LHS, we get 
$$
 t  \| S_{t} \mu \|_{\varphi(V)}   
    \leq \int_0^t \|S_u \mu\|_{\varphi(V)} \, \d u + b \frac{t^2}2  \|\mu\|.
$$
Together with the first inequality, this allows us to conclude. 
 \end{proof}

 \begin{proof}[Proof of Theorem \ref{thm:sHarris}]
   Thanks to Corolary~\ref{cor:weak-gen-Gronwall}, we see that the
   hypotheses of Theorem~\ref{thm:sHarris} are met for $V_2 = V$ and
   $V_1 = \psi(V)$ for any $\psi$ as in the statement of
   Theorem~\ref{thm:subgeometric_Harris_discrete}.  We may then apply
   Theorem \ref{th:sHarris-interpolated} and conclude.
 \end{proof}
 
 The above result has to be compared with the already known following convergence result.

\begin{thm}[subgeometric Harris]\label{th:sHarris-HairerVersion}  
  Consider a Feller type stochastic semigroup $(S_t)_{t \geq 0}$ on
  $\M_{V}$ which satisfies both the generator Lyapunov condition
  (Hypothesis \ref{hyp:generator_Lyapunov}) and the Harris
  irreducibility condition (Hypothesis \ref{hyp:Harris}).  There holds
  \beqn\label{eq:DFGrate} \| S_t \nu \| \lesssim {1 \over H^{-1}(t)}
  \| \nu \|_V, \quad \forall \, t \ge 0, \,\, \forall \, \nu \in
  \cN_V, \eeqn where $H$ is defined by
  $\displaystyle{ H (u) := \int_1^u {ds \over \varphi(s)}}.  $ It is
  worth observing that \bean
  {1 \over H^{-1}(t)} &\simeq&{t^{-k/\delta}} \quad\hbox{when}\quad m =  \langle x \rangle^k, \,\, \varphi(u) =  u^{1-\delta/k}, \,\, 0 < \delta < k; \\
  {1 \over H^{-1}(t)} &\simeq& e^{-\lambda t^{\sigma/(\sigma+\delta)}}
  \quad\hbox{when}\quad m = e^{ \langle x \rangle^\sigma}, \,\,
  \varphi(u) = {u \over (\log u)^{\delta/\sigma}}, \,\, \delta,\sigma
  >0, \eean when $\Omega:= \R^d$ and
  $\langle x \rangle := (1+|x|^2)^{1/2}$.
\end{thm}  

 It is worth emphasizing that the above rates of convergence are precisely the same as those obtained by our method for the same two examples presented  at the end of the Section~\ref{sec:subgeometric-discrete} as made explicit in \eqref{eq:HarrisSubgeo2-poly} and \eqref{eq:HarrisSubgeo2-expo}. 
  
 \smallskip
Theorem~\ref{th:sHarris-HairerVersion} has been established in \citet*{GDF2009} and an
alternative proof has been proposed in \citet{Hairer2016notes}.  Both
are based on non constructive probabilistic arguments that we do not
present here.  We mention however that the proof of
Theorem~\ref{th:sHarris-HairerVersion} as found in \citet{GDF2009,
  Hairer2016notes} consists in establishing
$$
\int_0^\infty \varphi(H^{-1}(s)) \| \nu_s \| \, ds \le  C\| \nu_0 \|_V,
$$
for any $\nu_0 \in \cN_V$  (in fact for $\nu_0 = \delta_x - \delta_y$). Because $s \mapsto \| \nu_s \| $ is decreasing and $(H^{-1})' = \varphi(H^{-1})$, one deduces
\bean
H^{-1}(t) \| \nu_t \| 
&\le& H^{-1}(t) \| \nu_t \|  -  \int_0^t  H^{-1}(s) ( {d \over ds} \|  \nu_s \| )  ds   
 \\
&=& H^{-1}(0) \| \nu_0 \|   +  \int_0^t  \varphi(H^{-1}(s))  \|  \nu_s \|    ds 
\\
&\le& H^{-1}(0) \| \nu_0 \| + 
C\| \nu_0 \|_V,
\eean
which is nothing but \eqref{eq:DFGrate}.

\smallskip

We have not been able to give a constructive deterministic proof of
Theorem~\ref{th:sHarris-HairerVersion}.  However, our analysis makes
it possible to recover Theorem~\ref{th:sHarris-HairerVersion} for
some specific but common examples,  as explained at the very
  end of Section~\ref{sec:subgeometric-discrete}. We remark that our
results give constructive constants in all cases, which is an
improvement in all subgeometric cases.


\section{Existence  of   an equilibrium under a
  subgeometric Lyapunov condition}
\label{sec:equilibrium}

We give here a quite general result about existence of an
 equilibrium for a Feller-type stochastic semigroup which is
independent of our previous results and in particular does not need  a coupling or Harris condition.

We thus consider hereafter a Feller type stochastic
semigroup $(S_t)_{t \geq 0}$ and we assume that the weak generator Lyapunov condition
(Hypothesis~\ref{hyp:weak_semigroup_Lyapunov}) holds for a weight function
$V \: \Omega \to [1,+\infty)$, a concave function $\varphi \: [1,+\infty) \to [1,+\infty)$, for which we may assume $\varphi' \le 1$ without lost of generality, 
and some   constants $b, \sigma > 0$. Introducing the constant $R := \sup V \in [1,\infty]$, 
we furthermore assume that 
$$
\varphi(R) > b/\varsigma \quad\hbox{ and }\quad \{ V \le \rho \} \ \hbox{is compact 
for any} \ \rho \in [1,R), 
$$
the last condition being fundamental in the present approach which is based on the use of the Prokhorov theorem about compactness of tight sequences. 
More precisely, from the last condition and the Prokhorov theorem, we may claim that any sequence $(\mu_n)$ of $\P$ with uniformly (in $n$) bounded $\varphi(V)$-moment is relatively compact
in $\P$. 

\smallskip
By fixing $ \rho \in [1,R)$ large enough and $\eps > 0$ small enough such that $(\varsigma - \eps) \varphi(\rho) \ge b$, 
we deduce that  
$$
L V_3 \le - V_2 + b  \, {\bf 1}_{\CC}, 
$$
where $L$ is the generator of the associated Markov-Feller semigroup $(P_t)$ on $C_0(\Omega)$, 
$V_3 := V$, $V_2 := \eps \varphi(V)$ and $\CC := \{x \in \Omega \mid V_2(x) \le \rho \}$.

\smallskip
The above Foster-Lyapunov condition provides a sufficient condition
for the existence of an equilibrium. 

\begin{thm}
  \label{sHarris:IM-existence}
  Any stochastic semigroup $(S_t)$ on $\M_{V}$ which fulfills the
  above Lyapunov condition has at least one invariant probability
  measure $\mu^* \in \M_{\varphi(V)}$.
\end{thm}

\begin{proof}[Proof of Theorem~\ref{sHarris:IM-existence}]
  {\sl Step 1. We prove that $(S_t)$ is bounded in the sense of Cesàro
    in $\M_{V_2}$. } 
    We
  define
  $$
  \AA := b \chi, \quad \BB := L - \AA,
  $$
  with $\chi \in C_0(\Omega)$ such that
  $\1_{\CC} \le \chi \le 1$.  Since $\BB$ is a bounded
  perturbation of $L$, we classically know that $\BB$ generates a
  semigroup $S_\BB$
  on the same space $C_0(\Omega)$
  and furthermore
  $$
 \BB \ge L - b,
 \quad \BB V_3 \le - V_2 \le 0.
  $$ 
  From the first inequality, we have $S_\BB(t) \ge e^{-bt} S_L(t) \ge 0$ for any $t \ge 0$, so that both $S_\BB$ and $S_\BB^*$ are positive semigroups.  
  Because of the Duhamel formula 
  $$
  S^*_{\BB} = S + S^*_\BB *(- \AA) S \le S,
  $$
and $S^*_\BB$ is a semigroup of contraction on $\M$. In particular  $S_{\BB} \in L^\infty_t (\BBB(\M))$,  where here and below, $L^\infty_t (\XX)$ denotes the space of bounded function from $\R_+$ into $\XX$. 
  From the same Duhamel formula, 
  we see that $S^*_\BB$ is well defined on $\M_{V_3}$ and has at least exponential growth rate. We can get a
  more accurate information. For $0 \le \mu_0$ in the domain of $S^*_\BB$ (defined in $\M_{V_3}$) and denoting $\mu_t := S^*_\BB(t) \mu_0$, we may compute 
    \begin{align*}
    \ddt \int \mu_t \, V_3 \le
    \int \mu_t \, \BB V_3 \le - \int \mu_t \, V_2, 
  \end{align*}
  so that 
  $$
   \int \mu_t \, V_3 + \int_0^t  \int \mu_s \, V_2  ds \le  \int \mu_0 \, V_3, \quad \forall \, t \ge 0.
   $$
  We deduce that  
  $$
  S^*_{\BB} \in L^\infty_t (\BBB(\M_{V_3})); \quad \int_0^\infty \|
  S^*_\BB(t) \mu_0 \|_{\M_{V_2}} \, dt \le \| \mu_0 \|_{\M_{V_3}}, \ \
  \forall \, \mu_0 \in \M_{V_3}.
  $$
We thus obtain  $S^*_{\BB} \in L^\infty_t (\BBB(\M_{V_2}))$, by interpolation together with the previous estimate  $S_{\BB} \in L^\infty_t (\BBB(\M))$. 
Alternatively, we could have used Lemma~\ref{lem:weak-gen-Gronwall}, in order to get
$$
\BB \varphi(V)  \le (-\varsigma \varphi(V) +b) \varphi'(V) - b \chi \varphi(V) \le b ( {\bf 1}_{\CC}  - \chi) \le 0, 
$$
next to compute directly 
$$
    \ddt \int (S^*_\BB(t) \mu_0) \, V_2  \le 0, 
$$
for $0 \le \mu_0$ in the domain (in $\M_{V_2}$) of $S^*_\BB$, and finally to deduce that $(S^*_\BB)$ is a semigroup of contractions in $\M_{V_2}$.
We next come back the splitting of the semigroup through the Duhamel
  formula
  $$
  S= S^*_\BB + S^*_\BB * \AA S,
  $$
  and we introduce the associated Cesàro means
  $$
  U_T := {1 \over T} \int_0^T S(t) \, dt, \quad V_T := {1 \over T}
  \int_0^T S^*_\BB(t) \, dt, \quad W_T := {1 \over T} \int_0^T (S^*_\BB *
  \AA S) (t) \, dt.
  $$
  We obviously have
  $$
  \| V_T \|_{\BBB(\M_{V_2})} \le {1 \over T} \int_0^T \| S^*_\BB(t)
  \|_{\BBB(\M_{V_2}) } \, dt \le 1.
  $$
  On the other hand, for $0 \le \mu_0 \in \M_{V_2}$, we have
  $$
  S^*_\BB(\tau) \int_0^{T-\tau} \AA \, S(s) \, \mu_0 \, ds \le
  S^*_\BB(\tau) \int_0^{T} \AA \, S(s) \, \mu_0 \, ds , \quad \forall \,
  T > \tau > 0,
  $$ 
  by positivity of the three operators involved in this integral
  formula, and then \bean \| W_T \mu_0 \|_{\M_{V_2}}
  &=& \Bigl\| {1 \over T} \int_0^T S^*_\BB(\tau) \int_0^{T-\tau} \AA \,
  S(s) \, \mu_0 \, d\tau ds \Bigr\|_{\M_{V_2}}
  \\
  &\le& {1 \over T} \int_0^\infty \Bigl\| S^*_\BB(\tau) \, \int_0^{T} \AA
  \, S(s) \, ds \mu_0 \Bigr\|_{\M_{V_2}} d\tau
  \\
  &\le& {1 \over T} \Bigl\| \int_0^T \AA \, S(s) \, ds \mu_0
  \Bigr\|_{\M_{V_3}} \le \| \AA \|_{\BBB(\M;\M_{V_3})} \| \mu_0 \|_{\M},
  \eean so that $W_T$ is uniformly bounded in
  $L_t^\infty(\BBB(\M_{V_2}))$. We then deduce that $U_T = V_T + W_T$ is
  also uniformly bounded in $L_t^\infty(\BBB(\M_{V_2}))$.

  \smallskip\noindent {\sl Step 2. Existence of an   invariant measure
    $\mu^* \in \M_{V_2}$.}  We define
  ${\Bbb K} :=\M_{V_2} \cap \P$ and we fix $\mu_0 \in {\Bbb K}$
  arbitrary. Because of Step~1, the sequence $\mu_T = U_T \mu_0$ is
  bounded in $ {\Bbb K}$.  By Prokhorov's theorem the embedding
  $\M_{V_2} \subset \M$ is compact, and hence there exists a
  subsequence $(\mu_{T_k})$ and $\mu^* \in {\Bbb K}$ such that
  $\mu_{T_k} \wto \mu^*$ in the weak-$*$ sense $\sigma(\M,C_0)$ as
  $k\to\infty$. For any fixed $s > 0$, we observe that \bean S(s)
  \mu^* - \mu^* &=& \lim_{k \to \infty} \Bigl\{ {1 \over T_k}
  \int_0^{T_k} S(s) S(t) \mu_0 - {1 \over T_k} \int_0^{T_k} S(t)
  \mu_0 \, dt \Bigr\}
  \\
  &=& \lim_{k \to \infty} \Bigl\{ {1 \over T_k} \int_{T_k}^{T_k+s}
  S(t) \mu_0 - {1 \over T_k} \int_0^s S(t) \mu_0 \, dt \Bigr\} =
  0, \eean so that $\mu^*$ is an  invariant measure.
\end{proof}


\section*{Acknowledgements}

J.~A.~C.~acknowledges the support of grant PID2020-117846GB-I00, the
research network RED2018-102650-T, and the María de Maeztu grant
CEX2020-001105-M from the Spanish government.

\bibliography{bibliography}

\end{document}